\documentclass{article}

\usepackage{geometry}
\usepackage[mathcal]{euscript}
\usepackage{hyperref}
\usepackage{amsmath,amssymb,amsthm}
\usepackage{mathtools}
\usepackage{graphicx}
\usepackage{dutchcal}
\usepackage{import}

\usepackage{newtxtext,newtxmath,helvet,courier}
\usepackage{fix-cm}
\DeclareMathAlphabet{\mathpzc}{OT1}{pzc}{m}{it}

\usepackage{xcolor}
\usepackage{cases}
\newtheorem{thm}{Theorem}[section]
\newtheorem{lemma}[thm]{Lemma}
\newtheorem{prop}[thm]{Proposition}
\newtheorem{cor}[thm]{Corollary}
\theoremstyle{remark}
\newtheorem{rem}[thm]{Remark}

\theoremstyle{definition}

\newtheoremstyle{Claim}{}{}{\itshape}{}{\itshape\bfseries}{:}{ }{#1}
\theoremstyle{Claim}

\newcommand{\Z}{{\mathbb{Z}}}

\newcommand{\T}{{\mathbb{T}^d}}

\renewcommand{\H}{\mathcal{H}}

\newcommand{\R}{\mathbb{R}}

\newcommand{\N}{\mathbb{N}}

\newcommand\sP{\mathpzc{P}}

\newcommand{\vertiii}[1]{{\left\vert\kern-0.25ex\left\vert\kern-0.25ex\left\vert #1 
    \right\vert\kern-0.25ex\right\vert\kern-0.25ex\right\vert}}

\newcommand{\eps}{\varepsilon}
\newcommand{\norm}[1]{\left\lVert#1\right\rVert}

\titlepage
\title{Maximal $L^q$-regularity for parabolic Hamilton-Jacobi equations and applications to Mean Field Games}
\date{\today}

\author{Marco Cirant and Alessandro Goffi}

\begin{document}

\maketitle

\begin{abstract} 
In this paper we investigate maximal $L^q$-regularity for time-dependent viscous Hamilton-Jacobi equations with unbounded right-hand side and superlinear growth in the gradient. Our approach is based on the interplay between new integral and H\"older estimates, interpolation inequalities, and parabolic regularity for linear equations. These estimates are obtained via a duality method \`a la Evans. This sheds new light on a parabolic counterpart of a conjecture by P.-L. Lions on maximal regularity for Hamilton-Jacobi equations, recently addressed  in the stationary framework by the authors. Finally, applications to the existence problem of classical solutions to Mean Field Games systems with unbounded local couplings are provided.
\end{abstract}

\noindent
{\footnotesize \textbf{AMS-Subject Classification} 35F21, 35K55, 35B65, 35Q89}. {\footnotesize }\\
{\footnotesize \textbf{Keywords}}. {\footnotesize Maximal $L^q$-regularity, H\"older regularity, Hamilton-Jacobi equations with unbounded right-hand side, Khardar-Parisi-Zhang equation, Riccati equation, Adjoint method, Mean-Field Games.}

\section{Introduction}

The purpose of this paper is two-fold. First, to establish maximal $L^q$-regularity for parabolic Hamilton-Jacobi equations with unbounded right-hand side of the form
\begin{equation}\label{hjb}\tag{HJ}
\begin{cases}
\partial_t u(x,t)-\Delta u(x,t)+H(x,Du(x,t))=f(x,t)&\text{ in }Q_T = \T \times (0,T)\ ,\\
u(x,0)=u_0(x)&\text{ in }\T
\end{cases}
\end{equation}
where $H$ has superlinear growth in the gradient variable 
and $f\in L^q(Q_T )$ for some $q>1$. Second, to apply these regularity results to prove the existence of classical solutions for a large class of second order Mean Field Games systems with local coupling, i.e.
\begin{equation}\label{mfgdef}\tag{MFG}
\begin{cases}
-\partial_t u-\Delta u+H(x,Du)=g(m(x,t))&\text{ in }Q_T\\
\partial_tm-\Delta m-\mathrm{div}(D_pH(x,Du)m)=0&\text{ in }Q_T \\
m(0) = m_0, \quad u(T)=u_T &\text{ in }\T.
\end{cases}
\end{equation}

{\bf Maximal regularity for \eqref{hjb}. }The problem of maximal $L^q$-regularity for \eqref{hjb} amounts to show that bounds on the right-hand side $f$ in $L^q(Q_T)$ imply bounds on individual terms $\partial_t u$, $\Delta u$ and $H(x,Du)$ in $L^q(Q_T)$ (see Theorem \ref{maxreg1} below for a more precise statement). This problem has been proposed for stationary Hamilton-Jacobi equations by P.-L. Lions in a series of seminars (see e.g. \cite{Napoli,LionsSeminar}). Under the assumption that
\[H(x,p) =  |p|^\gamma, \qquad \gamma>1,\]
he conjectured that maximal regularity holds provided that $q$ is above the threshold $d(\gamma-1)/\gamma$. The conjecture has been proved recently in \cite{CG4} via a refined Bernstein method, which unfortunately breaks down in the parabolic setting. Here, we are able to obtain parabolic maximal regularity via different methods, assuming that $q$ is above a certain (parabolic) threshold, that is $$q \ge (d+2)\frac{\gamma-1}{\gamma} = \frac{d+2}{\gamma'}$$ in the regime of {\it sub-natural growth}, that is for $\gamma < 2$. For $\gamma \ge 2$, we obtain the larger threshold $$q>(d+2)\frac{\gamma-1}2.$$
Besides maximal regularity, we prove new results on the H\"older regularity of solutions when $\gamma > 2$.

\smallskip

Before stating our results, we briefly review some contributions on the existence and regularity of solutions to Hamilton-Jacobi equations with unbounded right-hand side. We first discuss the case $\gamma \le 2$, that is when the nonlinear term $H(Du)$ plays a mild role, being the diffusion term (formally) dominating at small scales. When $f \in L^q$, $q > (d+2)/2$, boundedness and H\"older continuity of weak solutions is classical \cite[Chapter 5]{LSU}. For such values of $q$, boundedness has actually been established for much more general quasi-linear equations, see e.g. \cite{OP}. For these problems, it was shown \cite{Posteraro, Aglio} that existence of weak (and possibly unbounded) solutions holds up to $q = (d+2)/2$. For $\gamma < 2$, namely strictly below the natural growth, the existence assumption has been relaxed to $q \ge (d+2)/\gamma'$ in the recent paper \cite{Magliocca}. Still, solutions obtained are in a weak or renormalized sense, and though they need not be bounded, the question of further regularity beyond $H(Du) \in L^1$ in the existence regime $(d+2)/\gamma' \le q < (d+2)/2$ has remained open so far. Beyond the natural growth, that is in the super-quadratic case $\gamma > 2$, existence and regularity is even less understood. Under the sign condition $f \ge 0$, it has been proven in \cite{CSil,SV} that viscosity solutions enjoy H\"older bounds depending on $\|f\|_q$, with $q > 1+d/\gamma$. It is worth mentioning that these results only rely on the super-quadratic nature of $H$, and tolerate degenerate diffusions. Below this exponent, bounds in $L^p$ spaces have been shown in \cite{CGPT} (see also \cite[Section 3.4]{Gomesbook}). It is remarkable that the exponent $1+d/\gamma$ decreases as $\gamma$ grows in the super-quadratic regime, so estimates in $L^\infty$ require milder assumptions on $f$ than in the sub-quadratic regime. On the other hand, further regularity, especially involving $Du$, seems to be difficult to achieve. For general $\gamma > 1$, Lipschitz regularity has been recently investigated in \cite{CG2}, and proven under the (optimal) assumption $q > d+2$ when $\gamma \le 3$ and larger $q$ when $\gamma > 3$. Note that whenever Lipschitz regularity is established, maximal regularity for \eqref{hjb} follows by maximal regularity for linear equations (e.g.  \cite{Lamberton, HP,DHP,PS,PrussSimonett}), being $H(Du)$ controlled in $L^\infty$. We finally mention that within the context of $L^p$-viscosity solutions, Hamilton-Jacobi equations with unbounded right-hand side have been considered, see e.g. \cite{CKS}.

Maximal regularity typically requires some mild smoothness of coefficients in the equation (diffusions with coefficients merely in $L^\infty$ are not allowed even in the linear framework), but provide integrability of $D^2 u, \partial_t u$, thus allowing to recover most of the aforementioned properties of solutions by means of Sobolev embeddings, as $q$ varies. To our knowledge, there are only a few instances of this type of results in the literature, and in the regime $\gamma \le 2$ only. In \cite{SW} maximal regularity is stated for $f\in L^q$, $q\geq d+1$ (see also \cite{S1,S2}). For Hamilton-Jacobi equations driven by the Laplacian, some results can be found in \cite{Gomesbook}, under the more general assumption $q > (d+2)/2$, but no results are available below this exponent, nor for $\gamma > 2$. 

\smallskip

As we previously announced, a first main result of this paper consists in achieving maximal regularity in the full range $q > (d+2)/\gamma'$ in the sub-quadratic setting $\gamma \le 2$, and for $q > (d+2)(\gamma-1)/2$ when $\gamma > 2$. Data will be periodic in the $x$-variable, namely functions will be defined over the $d$-dimensional flat torus $\T$. This will be convenient for the applications to MFG systems. We suppose that $H \in C(\T \times \R^d)$ is convex in the second variable, and that there exist constants $\gamma > 1$ and $C_H>0$ such that
\begin{equation}\label{H}\tag{$H$}
C_H^{-1}|p|^{\gamma}-C_H\leq H(x,p)\leq C_H(|p|^{\gamma}+1)\ ,
\end{equation}
for every $x\in \T$, $p\in\R^d$. 
 Concerning the case $\gamma \ge 2$, we will further assume some additional regularity in the $x$-variable, i.e for $\alpha \in (0,1)$ to be specified,
\begin{equation}\label{Ha}\tag{$H_\alpha$}
H(x, p) - H(x+\xi, p)  \le C_H |\xi|^\alpha \big( |D_pH(x,p)|^{\gamma'}+1 \big)
\end{equation}
for all $x,\xi \in \T$ and $p \in \R^d$. A typical example of $H$ satisfying \eqref{H} is
\[
H(x,p) = h(x)|p|^{\gamma} + b(x) \cdot p, \qquad 0 < h_0 \le h(x), \quad h, b \in C(\T).
\]
If $h \in C^\alpha(\T)$, this Hamiltonian will satisfy also \eqref{Ha}.

Hoping to help the reader to have a clearer picture, we sketch  known and new regularity regimes as $\gamma$ and $q$ vary in Figure \ref{hjbplot}.

\begin{thm}\label{maxreg1}
Assume that \eqref{H} holds, and \eqref{Ha} also when $\gamma \ge 2$ (with $\alpha$ as in Theorem \ref{mainholder} below). Let $u \in W^{2,1}_q(Q_T)$ be a strong solution to \eqref{hjb} and assume that for some $K > 0$
\[
\|f\|_{L^q(Q_T)}+\|u_0\|_{W^{2-\frac{2}{q},q}(\T)} \le K.
\]
If
\[
q > 
\begin{cases}
(d+2)\frac{\gamma-1}{\gamma} & \text{if $1 + \frac{2}{d+2} < \gamma < 2$} \medskip \\ 
(d+2)\frac{\gamma-1}{2} & \text{if $\gamma \ge 2$}
\end{cases}
\]
then, there exists a constant $C>0$ depending on $K, q, d, C_H, T$ such that
\[
\|u\|_{W^{2,1}_q(Q_T)}+\|Du\|_{L^{\gamma q}(Q_T)} \leq C.
\]
\end{thm}

The strategy of the proof is based on the following procedure. By maximal regularity for linear equations, one has
\[
\|D^2 u\|_{L^q} \lesssim \|H(Du)\|_{L^q} +\|f\|_{L^q} + \|u_0\|_{W^{2-2/q,q}} \lesssim \|Du\|^\gamma_{L^{\gamma q}} +\|f\|_{L^q} + \|u_0\|_{W^{2-2/q,q}}.
\]
Then, one looks for a suitable norm $\vertiii{\cdot}$ so that a Gagliardo-Nirenberg type interpolation inequality of the form
\[
\|Du\|^\gamma_{L^{\gamma q}} \lesssim \|D^2 u \|_{L^q}^{\gamma \theta} \ \vertiii{u}^{\gamma(1-\theta)}
\]
with $\gamma \theta < 1$ holds. Combining the two inequalities, maximal regularity is achieved whenever it is possible to produce bounds on $\vertiii{u}$. This way of producing estimates for nonlinear problems, using linear estimates and interpolation inequalities, goes back to the works of Amann and Crandall \cite{AC}. When $\gamma = 2$, a good choice for $\vertiii{\cdot}$ is the $L^\infty$ norm, which is easily controlled by $\|f\|_\infty$ (ABP type estimates allow to reach $\|f\|_{d+1}$, as in \cite{SW}).  Here, we start with the observation that when $\gamma < 2$, the optimal choice is a suitable $L^p$ norm, while for $\gamma > 2$, a H\"older bound on $u$ is needed. A crucial step in this work is the derivation of such estimates. These are obtained by duality arguments, inspired by \cite{Evans,CG2}. The main idea behind duality is to shift the attention from \eqref{hjb} to the formal adjoint of its linearization, which is a Fokker-Planck equation. Crossed information on $u$ and the solution $\rho$ to the ``dual'' equation allow to retrieve estimates on $\rho$, that are then transferred back to $u$. A more detailed heuristic explanation of this method can be found in \cite{CG2} (see also references therein). Note that \cite{CG2} is devoted to Lipschitz regularity of $u$, while here we investigate H\"older regularity, which requires a further inspection of the regularity of $\rho$ at the level of Nikol'skii spaces.

We underline that our results on H\"older regularity in the super-quadratic case are new in the following sense. Recall that H\"older bounds have been obtained in \cite{CSil,SV} when $q > 1+ d/\gamma$, while here we assume the stronger requirement $q > (d+2)/\gamma'$ (which is the natural one for maximal regularity). In \cite{CSil,SV} sign assumptions on $f$ are in force, and explicit H\"older exponents are not provided. Here, we do not require any assumption on the sign of $f$, and produce explicit H\"older exponents. The statement is as follows.

\begin{thm}\label{mainholder}
Assume that \eqref{H} and \eqref{Ha} hold, with $\gamma \ge 2$. Let $u$ be a strong solution to \eqref{hjb} in $W^{2,1}_q(Q_T)$, $q > \frac{d+2}{\gamma'}$. Then, there exists a positive constant ${C}$ (depending on $\|u_0\|_{C^\alpha(\T)}, \|f\|_{L^{q}(Q_T)}, H, q, d, T$) such that
\begin{equation}\label{abound}
\sup_{t \in [0,T]}\|u(t)\|_{C^\alpha(\T)} \le C,
\end{equation}
where $\alpha = \gamma'-\frac{d+2}{q}$ if $q < \frac{d+2}{\gamma'-1}$, while $\alpha \in (0,1)$ if $q \ge \frac{d+2}{\gamma'-1}$.
\end{thm}

We mention that the proof of the H\"older bounds could be localized in time, thus assuming merely $u_0 \in C(\T)$. The constant $C$ will then depend on $\|u\|_\infty$, see Remark \ref{holder2}.
\smallskip

In the sub-quadratic regime $\gamma < 2$, the existence and uniqueness of weak solutions can be obtained up to the critical integrability exponent $q = (d+2)\frac{\gamma-1}{\gamma}$  \cite{Magliocca}. From the viewpoint of maximal regularity, this endpoint situation is a bit more delicate than the one treated in Theorem \ref{maxreg1}, which concerns $q > (d+2)\frac{\gamma-1}{\gamma}$. Indeed, the heuristic procedure previously discussed yields
\[
\|D^2 u\|_{L^q} \lesssim \|D^2 u \|_{L^q} \ \|u\|_{L^p}^{\gamma-1} +\|f\|_{L^q} + \|u_0\|_{W^{2-2/q,q}},
\]
which is meaningful only if $ \|u\|_{L^p}$ is small. We circumvent this issue by shifting the analysis from $u$ to $u-u_k$, where $u_k$ is the solution to a suitable regularized problem. Crucial stability estimates on $ \|u-u_k\|_{L^p}$ then lead to the second main maximal regularity result of the paper.

\begin{thm}\label{maxregc}
Assume that \eqref{H} holds, and $1 + \frac{2}{d+2} < \gamma < 2$. Let $u \in W^{2,1}_q(Q_T)$ be a strong solution to \eqref{hjb}, and
\[
q = (d+2)\frac{\gamma-1}{\gamma}.
\]
Then, there exists a constant $C>0$ depending on $f, \|u_0\|_{W^{2-2/q,q}}, q, d, C_H, T$ such that
\[
\|u\|_{W^{2,1}_q(Q_T)}+\|Du\|_{L^{\gamma q}(Q_T)} \leq C.
\]
\end{thm}

We stress that in this limiting case the dependance of the constant $C$ with respect to $f$  is not just through its norm, as in the previous Theorem \ref{maxreg1}, but on subtler properties; see Remark \ref{remlim} below for additional details.

Further comments concerning the thresholds for $q$ in Theorems \ref{maxreg1} and  \ref{maxregc} are now in order. 
%When $\gamma \le 2$ and in the critical case $q = (d+2)/\gamma'$, we can prove that a control of the norms of $u$ in terms of $\|f\|_{L^q(Q_T)}+\|u_0\|_{W^{2-\frac{2}{q},q}(\T)}$ holds, provided that $\|f\|_{L^q(Q_T)}+\|u_0\|_{W^{2-\frac{2}{q},q}(\T)}$ is small enough (see Remark \ref{remlim}). \textcolor{red}{We believe that one could drop this smallness assumption, and obtain bounds on norms of $u$ that are stable for $f$ belonging to equi-integrable sets in $L^q$ (as it happens for $L^p$ estimates on $u$ \cite{Magliocca})}. 
When $q < (d+2)/\gamma'$ we believe that maximal regularity results are false in general. This would be in line with known results for the associated stationary problem \cite{CG4}, for which maximal regularity does not hold when $q < d/\gamma'$. We also believe that the endpoint $q=\frac{d+2}2$ for quadratic problems $\gamma = 2$ could be treated by our methods, using more refined interpolation estimates in Orlicz spaces (or the analysis of a linear equation obtained via the Hopf-Cole transformation). Regarding the super-quadratic regime $\gamma > 2$, we do not know whether our assumption $q > (d+2)(\gamma-1)/2$ is optimal or not.

It is finally worth noticing that our results, and in particular the H\"older estimates in Theorem \ref{mainholder}, apply to equations with repulsive gradient term (e.g. Kadar-Parisi-Zhang type PDEs), i.e.
\[
\partial_t v-\Delta v=G(x,Dv(x,t))-f(x,t)
\]
with $G$ satisfying \eqref{H}. In other words, the sign in front of $H$ and $f$ does not matter in \eqref{hjb}, since it is sufficient to observe that $u(x,t)=-v(x,t)$ solves \eqref{hjb} with $H(x,p)=G(x,-p)$, which satisfies \eqref{H} too. We refer to \cite{AP, baSW} and the references therein for further discussions on these equations.

\medskip 

{\bf Mean Field Games. } Armed with maximal regularity results for Hamilton-Jacobi equations, we describe our contributions on Mean Field Games (MFG) systems of the form \eqref{mfgdef}. These arise in the MFG framework, which is a class of methods inspired by ideas in statistical physics to study differential games with a population of infinitely many indistinguishable players \cite{ll,lltime}. The PDE system \eqref{mfgdef} appears naturally when the running cost of a single player depends on the population's density in a pointwise manner. This results in the nonlinear coupling term $g(m(x,t))$ between the Hamilton-Jacobi and the Fokker-Planck equation in \eqref{mfgdef}. When $g(\cdot)$ is an unbounded function, regularity of solutions is still a challenging problem that has not been solved in full generality. 

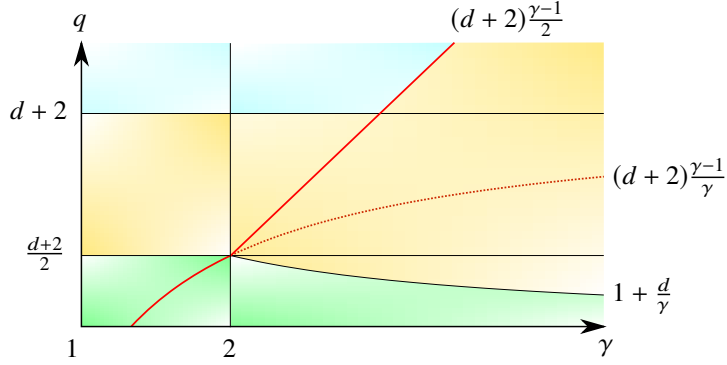
\begin{figure}
\centering
\def\svgwidth{.8\columnwidth}
\begingroup%
  \makeatletter%
  \providecommand\color[2][]{%
    \errmessage{(Inkscape) Color is used for the text in Inkscape, but the package 'color.sty' is not loaded}%
    \renewcommand\color[2][]{}%
  }%
  \providecommand\transparent[1]{%
    \errmessage{(Inkscape) Transparency is used (non-zero) for the text in Inkscape, but the package 'transparent.sty' is not loaded}%
    \renewcommand\transparent[1]{}%
  }%
  \providecommand\rotatebox[2]{#2}%
  \newcommand*\fsize{\dimexpr\f@size pt\relax}%
  \newcommand*\lineheight[1]{\fontsize{\fsize}{#1\fsize}\selectfont}%
  \ifx\svgwidth\undefined%
    \setlength{\unitlength}{569.25961953bp}%
    \ifx\svgscale\undefined%
      \relax%
    \else%
      \setlength{\unitlength}{\unitlength * \real{\svgscale}}%
    \fi%
  \else%
    \setlength{\unitlength}{\svgwidth}%
  \fi%
  \global\let\svgwidth\undefined%
  \global\let\svgscale\undefined%
  \makeatother%
  \begin{picture}(1,0.39015309)%
    \lineheight{1}%
    \setlength\tabcolsep{0pt}%
    \put(0,0){\includegraphics[width=\unitlength,page=1]{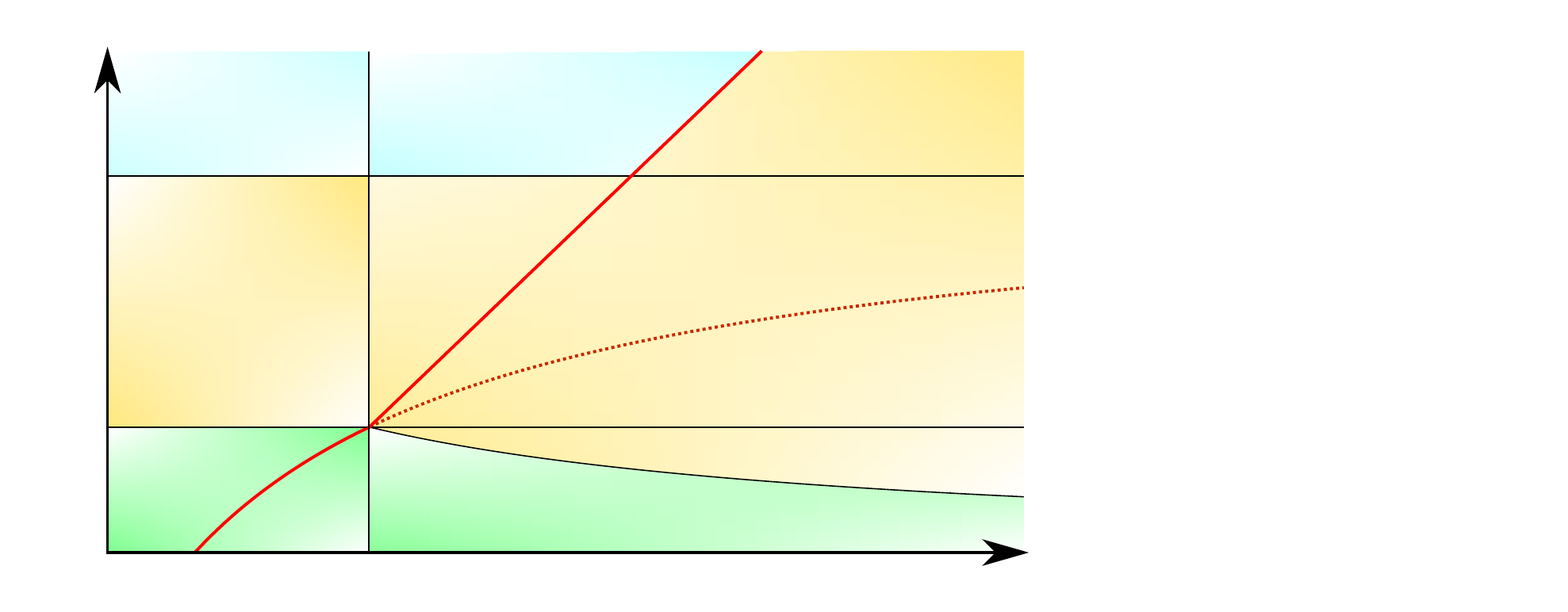}}%
    \put(0.05078519,0.00318598){\color[rgb]{0,0,0}\makebox(0,0)[lt]{\lineheight{1.25}\smash{\begin{tabular}[t]{l}$1$\end{tabular}}}}%
    \put(0.22665792,0.00326416){\color[rgb]{0,0,0}\makebox(0,0)[lt]{\lineheight{1.25}\smash{\begin{tabular}[t]{l}$2$\end{tabular}}}}%
    \put(0.6624429,0.06605718){\color[rgb]{0,0,0}\makebox(0,0)[lt]{\lineheight{1.25}\smash{\begin{tabular}[t]{l}$1+\frac{d}{\gamma}$\end{tabular}}}}%
    \put(0.66322217,0.19808974){\color[rgb]{0,0,0}\makebox(0,0)[lt]{\lineheight{1.25}\smash{\begin{tabular}[t]{l}$(d+2)\frac{\gamma-1}{\gamma}$\end{tabular}}}}%
    \put(0.48208806,0.37522273){\color[rgb]{0,0,0}\makebox(0,0)[lt]{\lineheight{1.25}\smash{\begin{tabular}[t]{l}$(d+2)\frac{\gamma-1}{2}$\end{tabular}}}}%
    \put(0.64703665,0.00953545){\color[rgb]{0,0,0}\makebox(0,0)[lt]{\lineheight{1.25}\smash{\begin{tabular}[t]{l}$\gamma$\end{tabular}}}}%
    \put(0.0589434,0.37624423){\color[rgb]{0,0,0}\makebox(0,0)[lt]{\lineheight{1.25}\smash{\begin{tabular}[t]{l}$q$\end{tabular}}}}%
    \put(0.00455607,0.11204171){\color[rgb]{0,0,0}\makebox(0,0)[lt]{\lineheight{1.25}\smash{\begin{tabular}[t]{l}$\frac{d+2}2$\end{tabular}}}}%
    \put(-0.01190795,0.27203785){\color[rgb]{0,0,0}\makebox(0,0)[lt]{\lineheight{1.25}\smash{\begin{tabular}[t]{l}$d+2$\end{tabular}}}}%
  \end{picture}%
\endgroup%
\caption{\footnotesize A sketch of estimates that are known for solutions to \eqref{hjb}, as $\gamma$ and $q$ vary: $L^p$ (green region), H\"older (orange region), Lipschitz (light blue region). Above the solid red line we prove here maximal regularity estimates. Above the dashed red line we prove here new H\"older estimates.}\label{hjbplot}
\end{figure}

\smallskip

In the so-called {\it monotone} case, that is when $g(\cdot)$ is increasing, there are basically no restrictions on the growth of $f$ if one looks for solutions in some weak sense \cite{CGPT, Porr}, or for classical solutions when the time horizon $T$ is small \cite{Ambrose, CGM}. Classical solutions for arbitrary $T$ can be obtained by requiring a mild behaviour of $g(m)$ as $m \to \infty$, or a mild behaviour of $H(p)$, i.e. $\gamma \le 1 + 1/(d+1)$ (as suggested in \cite{LionsSeminar}). We deal here with the first situation. For the model problem $g(m) = m^r$, an upper bound on $r$ depending on $\gamma$ and the space dimension $d$ is usually required. Concerning the subquadratic case $\gamma < 2$, a main reference is \cite{Gomessub} (see also \cite{BBF}), while \cite{Gomessup} explores the superquadratic case $2 < \gamma < 3$. 

We will assume here that $f:\T\times[0,\infty)\to\R$ is of class $C^1$, and that there exist $r>0$ and $C_g > 0$ such that
\begin{equation}\label{fde}\tag{$M^+$}
C^{-1}_g m^{r-1}  \le g'(m) \le C_g(m^{r-1} + 1) \qquad \forall m \ge 0.
\end{equation}
This implies that $g(\cdot)$ is monotone increasing, and bounded from above and below by power-like functions of type $m^r$. As for the Hamiltonian, we will further assume some additional smoothness and uniform convexity in the $p$-variable, having always in mind a power-like growth $|p|^\gamma$: for every $p \in \R^d, M\in\mathrm{Sym}_d$
\begin{equation}\label{H2}\tag{$H2$}
\begin{gathered}
\mathrm{Tr}(D^2_{pp}H(x,p)M^2) \geq C_H^{-1}(1+|p|^2)^{\frac{\gamma-2}2}|M|^2 -C_H, \\
|D^2_{px}H(x,p)| \leq C_H(|p|^{\gamma-1} + 1), \\
|D^2_{xx}H(x,p)| \leq C_H (|p|^\gamma + 1).
\end{gathered}
\end{equation}

Our main result on this class of problems reads as follows.

\begin{thm}\label{mfg1}
Under the assumptions \eqref{H}, \eqref{H2} and \eqref{fde}, there exists a (unique) smooth solution to \eqref{mfgdef} if
\begin{equation}
r < 
\begin{cases}
\frac{\gamma'}{d-2}\frac{d}{(d+2-\gamma')} & \text{if $1 + \frac{1}{d+1} < \gamma \le 2$} \medskip \\ 
\frac{2}{d(\gamma-1) - 2} & \text{if $\gamma \ge 2$}.
\end{cases}
\end{equation}
\end{thm}

It is understood here that there are no restrictions on $r$ when $d = 1,2$. To our knowledge, the results of this manuscript extend known classical regularity regimes in the sub-quadratic case. 
Note that
\begin{align*}
 \frac{\gamma'}{d-2}\frac{d}{(d+2-\gamma')} & \to +\infty \quad \text{as $\gamma \to 1 + \frac{1}{d+1}$}, \\
 \frac{\gamma'}{d-2}\frac{d}{(d+2-\gamma')} & \to \frac{2}{d-2} \quad \text{as $\gamma \to 2$}.
\end{align*}
Regarding the super-quadratic case, we actually obtain the same restriction $r < \frac{2}{d(\gamma-1) - 2}$ as in \cite{Gomessup}, but we are able to cover the full interval $\gamma \in [2, \infty)$, thus unlocking some smoothness regimes for $\gamma > 3$. Still, one may conjecture that, for all $d \ge 1$, no restrictions on $r$ should be required to get classical solutions (as for the purely quadratic case $H(p) = |p|^2$ investigated in \cite[Theorem 4.1]{NHM}), but this remains an open question.

\smallskip
In the {\it non-monotone} framework, the assumption of increasing monotonicity of $g(\cdot)$ is dropped. Conversely, one may even consider a $g(\cdot)$ which is strictly decreasing (focusing case), and model aggregation phenomena as in \cite{CeCi, CirJDE, CT}. In this framework, a control on the growth of $g$ is structurally needed. For stationary problems, it was shown in \cite{CCPDE} that when $g(m) = -m^r$, no solutions on $\R^d$ exist when $r > \frac{\gamma'}{d-\gamma'}$, and even for $r \ge \frac{\gamma'}{d}$ some further assumptions might be needed for solutions to be obtained. On the other hand, when $r < \frac{\gamma'}{d}$, existence of weak solutions is shown in \cite{CT}, but their classical regularity is proven under much stronger hypotheses, that impose $\gamma < 2$. 

Now, we suppose
\begin{equation}\label{fm2}\tag{$M^-$}
- C_g(m^r+1) \leq g(m)\leq C_g(m^r+1).
\end{equation}
With respect to the previous assumptions \eqref{fde} in the monotone case, $g$ need not be monotone increasing; in contrast, it can be strictly decreasing, for example $g(m) = -m^r$. We have the following
\begin{thm}\label{mfg2}
Under the assumptions \eqref{H} and \eqref{fm2}, there exists a smooth solution to \eqref{mfgdef} if
\begin{equation}
r < 
\begin{cases}
\frac{\gamma'}{d} & \text{if $1 + \frac{1}{d+1} < \gamma \le 2$} \medskip \\ 
\frac{2}{(d+2)(\gamma-1)-2} & \text{if $\gamma \ge 2$}.
\end{cases}
\end{equation}
\end{thm}
In other words, we prove that solutions in the focusing case obtained in \cite{CT} are always classical in the subquadratic regime. In the superquadratic regime, a new class of smooth solutions is obtained. We stress that the threshold $\frac{\gamma'}d$ is crucial, as one may have non-existence of solutions for large time-horizon $T$ when $r \ge \frac{\gamma'}d$ \cite{CD}. For a summary sketch of the existence regimes that we obtain here, see Figure \ref{mfgplot}.

Theorems \ref{mfg1} and \ref{mfg2} are proven as follows. First, we use structural estimates to deduce some a priori bounds on $\|g(m)\|_{L^q}$. These are mainly second-order type estimates in the monotone framework, and first-order estimates in the non-monotone one (plus some additional interpolation procedure). The assumptions on the growth $r$ of $g(\cdot)$ guarantee that $q$ is large enough to apply our maximal regularity results for the HJ equation. Then, once the crucial a priori bounds on $u$ in $W^{2,1}_q$ are established, a bootstrap procedure allows to get estimates up to second order derivatives in H\"older spaces, and existence follows via standard methods.

\begin{figure}
\centering
\def\svgwidth{.7\columnwidth}
\begingroup%
  \makeatletter%
  \providecommand\color[2][]{%
    \errmessage{(Inkscape) Color is used for the text in Inkscape, but the package 'color.sty' is not loaded}%
    \renewcommand\color[2][]{}%
  }%
  \providecommand\transparent[1]{%
    \errmessage{(Inkscape) Transparency is used (non-zero) for the text in Inkscape, but the package 'transparent.sty' is not loaded}%
    \renewcommand\transparent[1]{}%
  }%
  \providecommand\rotatebox[2]{#2}%
  \newcommand*\fsize{\dimexpr\f@size pt\relax}%
  \newcommand*\lineheight[1]{\fontsize{\fsize}{#1\fsize}\selectfont}%
  \ifx\svgwidth\undefined%
    \setlength{\unitlength}{632.1134932bp}%
    \ifx\svgscale\undefined%
      \relax%
    \else%
      \setlength{\unitlength}{\unitlength * \real{\svgscale}}%
    \fi%
  \else%
    \setlength{\unitlength}{\svgwidth}%
  \fi%
  \global\let\svgwidth\undefined%
  \global\let\svgscale\undefined%
  \makeatother%
  \begin{picture}(1,0.55400464)%
    \lineheight{1}%
    \setlength\tabcolsep{0pt}%
    \put(0,0){\includegraphics[width=\unitlength,page=1]{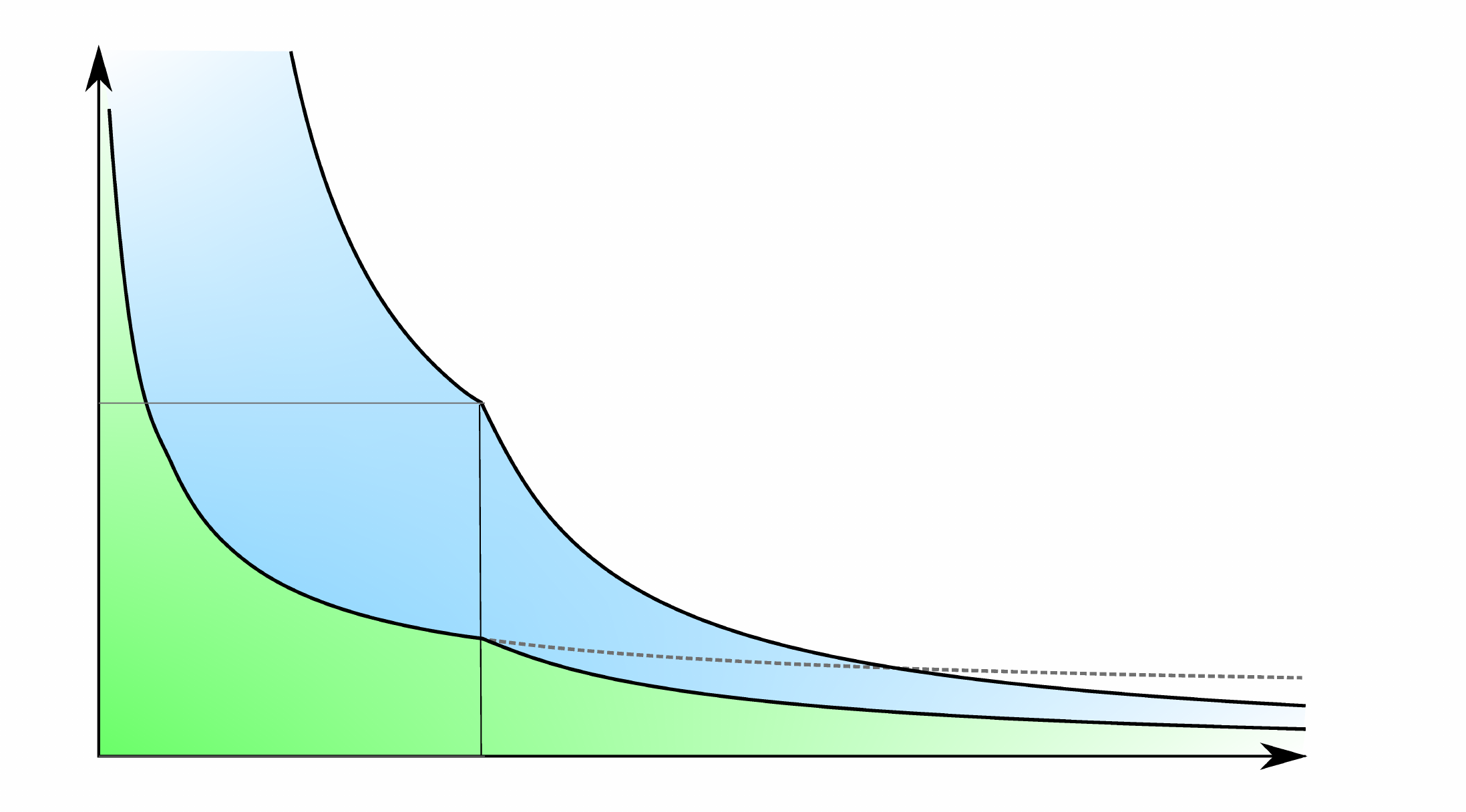}}%
    \put(0.31715705,0.00443404){\makebox(0,0)[lt]{\lineheight{1.25}\smash{\begin{tabular}[t]{l}$2$\end{tabular}}}}%
    \put(-0.00167624,0.2732){\makebox(0,0)[lt]{\lineheight{1.25}\smash{\begin{tabular}[t]{l}$\frac{d}{d-2}$\end{tabular}}}}%
    \put(0.02942546,0.04475034){\makebox(0,0)[lt]{\lineheight{1.25}\smash{\begin{tabular}[t]{l}$0$\end{tabular}}}}%
    \put(0.06841066,0.00390401){\makebox(0,0)[lt]{\lineheight{1.25}\smash{\begin{tabular}[t]{l}$1$\end{tabular}}}}%
    \put(0.87974343,0.00943889){\makebox(0,0)[lt]{\lineheight{1.25}\smash{\begin{tabular}[t]{l}$\gamma$\end{tabular}}}}%
    \put(0.74349057,0.11958927){\makebox(0,0)[lt]{\lineheight{1.25}\smash{\begin{tabular}[t]{l}$\frac{d}{\gamma'}$\end{tabular}}}}%
    \put(0.04971007,0.53696108){\makebox(0,0)[lt]{\lineheight{1.25}\smash{\begin{tabular}[t]{l}$r$\end{tabular}}}}%
    \put(0,0){\includegraphics[width=\unitlength,page=2]{plotmfg.pdf}}%
    \put(0.14084608,0.00442802){\makebox(0,0)[lt]{\lineheight{1.25}\smash{\begin{tabular}[t]{l}$1+\frac{1}{d+1}$\end{tabular}}}}%
  \end{picture}%
\endgroup%
\caption{\footnotesize A sketch of regions where existence of classical solutions to \eqref{mfgdef} are obtained, in the monotone case (light blue and green regions) and in the non-monotone case (green region only).}\label{mfgplot}
\end{figure}

\medskip

{\bf Outline}. In Sections \ref{stat} and \ref{par} we introduce the functional setting that is used throughout the paper, while Section \ref{regadj} is devoted to Sobolev regularity results for Fokker-Planck equations. Section \ref{firstbounds} comprehends the proof of a priori integral and H\"older estimates for Hamilton-Jacobi equations with unbounded ingredients. In Section \ref{sec;max}, Theorem \ref{maxreg1} on maximal regularity for \eqref{hjb} is proven, while in Section \ref{mfgs} the existence Theorems \ref{mfg1} and \ref{mfg2} for \eqref{mfgdef} are proven.

\bigskip

{\bf Acknowledgements.} The authors are members of the Gruppo Nazionale per l'Analisi Matematica, la Probabilit\`a e le loro Applicazioni (GNAMPA) of the Istituto Nazionale di Alta Matematica (INdAM). This work has been partially supported by 
the Fondazione CaRiPaRo
Project ``Nonlinear Partial Differential Equations:
Asymptotic Problems and Mean-Field Games". The second-named author wishes to thank Prof. Alessandra Lunardi for discussions about Lemma \ref{embprel}.

\section{Preliminaries}\label{funcspa}

\subsection{Fractional spaces of periodic functions}\label{stat}
Let $\mu\in(0,1)$ and $1\leq p,q\leq \infty$. The Besov space $B_{pq}^\mu(\T)$ consists of all functions $u\in L^p(\T)$ such that
the norm
\[
\|u\|_{B_{pq}^\mu(\T)}:=\|u\|_{L^p(\T)}+\left(\int_{\T}\frac{\|f(x+h)-f(x)\|^q_{L^p(\T)}}{|h|^{d+\mu q}}\,dh\right)^{\frac{1}{q}}
\]
is finite. When $p=q=\infty$ and $\mu=\alpha\in(0,1)$, the space $B_{\infty\infty}^\mu(\T)\simeq C^{\alpha}(\T)$ and it is endowed with the equivalent norm
\[
\|u\|_{C^\alpha(\T)}:=\|u\|_{C(\T)}+\sup_{x\neq y\in\T}\frac{|u(x)-u(y)|}{\mathrm{dist}(x,y)^\alpha}
\]
where $\mathrm{dist}(x,y)$ is the geodesic distance on $\T$. When, instead, $p=q\in(1,\infty)$ and $\mu$ is not an integer, it is immediate to recognize that $B_{pp}^\mu(\T)\simeq W^{\mu,p}(\T)$, where $W^{\mu,p}(\T)$ is the classical Sobolev-Slobodeckii scale in the periodic setting. When $q=\infty$, the space $B_{p\infty}^\mu(\T)\simeq N^{\mu,p}(\T)$ is known as Nikol'skii space \cite{Nik} and the above norm is interpreted as usual as
\[
\|u\|_{N^{\mu,p}(\T)}:=\|u\|_{L^p(\T)}+\underbrace{\sup_{h}|h|^{-\mu}\|u(x+h)-u(x)\|_{L^p(\T)}}_{[u]_{N^{\mu,p}(\T)}}\ ,
\]
see \cite[Chapter 17]{Leoni} for analogous spaces defined on $\R^d$, and \cite[p. 460]{T}, \cite[Section 3.5.4]{ST} for details in the periodic case.
For general $\mu>0$, let $\mu=k+\sigma>0$ with $k\in\N\cup\{0\}$, $\sigma\in(0,1]$, $1\leq p<\infty$. Then, we define the Nikol'skii class
\[
N^{\mu,p}(\T):=\{u\in W^{k,p}(\T):[D^\alpha u]_{N^{\sigma,p}(\T)}<\infty\ ,|\alpha|=k\}
\]
We mention that in view of \cite[Corollary 2 p. 143]{T92}, Nikol'skii spaces $N^{\mu,p}(\T)$ can be endowed with equivalent norms of the form
\[
\|u\|_{N^{\mu,p}(\T)}:=\|u\|_{L^p(\T)}+\sum_{|\alpha|=k}\sup_{|h|>0}|h|^{k-\mu}\|\Delta^r_h u\|_{L^p(\T)}
\]
where $k,r\in\Z$ are such that $0\leq k<\mu$ and $r>s-k$, where $\Delta^r_h$ are the increments of order $r$ and step $h$, see \cite{Leoni}. Moreover, these spaces can be characterized via real interpolation: for $m\in\N$, $p,q\in[1,\infty]$ and $\theta\in(0,1)$ we have\footnote{We denote here with $(X_0,X_1)_{\theta,q}$ the standard real interpolation space between Banach spaces $X_0, X_1$.}
\[
(L^p(\T),W^{m,p}(\T))_{\theta,q}\simeq B^{\theta m}_{p,q}(\T)\ ,
\]
with equivalence of the respective norms, see e.g. \cite{Lunardi}, \cite[Theorem 17.24]{Leoni}.

\subsection{Space-time anisotropic spaces}\label{par}
For any time interval $(t_1, t_2) \subseteq \R$, let $Q_{t_1, t_2}:=\T\times (t_1, t_2)$. We will also use the notation $Q_{t_2}:= \T\times (0, t_2)$. For any $p\geq1$ and $Q = Q_{t_1, t_2}$, we denote by $W^{2,1}_p(Q)$ the space of functions $u$ such that $\partial_t^{r}D^{\beta}_xu\in L^p(Q)$ for all multi-indices $\beta$ and $r$ such that $|\beta|+2r\leq  2$, endowed with the norm
\begin{equation*}
\norm{u}_{W^{2,1}_p(Q)}=\left(\iint_{Q}\sum_{|\beta|+2r\leq2}|\partial_t^{r}D^{\beta}_x u|^pdxdt\right)^{\frac1p}.
\end{equation*}
The space $W^{1,0}_p(Q)$ is defined similarly, and is endowed with the norm
\[
\norm{u}_{W^{1,0}_p(Q)}:=\norm{u}_{L^p(Q)}+\sum_{|\beta|=1}\norm{D_x^{\beta}u}_{L^p(Q)}\ .
\]
We define the space $\H_p^{1}(Q)$ as the space of functions $u\in W^{1,0}_p(Q)$ with $\partial_tu\in (W^{1,0}_{p'}(Q))'$, equipped with the norm
\begin{equation*}
\norm{u}_{\mathcal{H}_p^{1}(Q)}:=\norm{u}_{W^{1,0}_p(Q)}+\norm{\partial_tu}_{(W^{1,0}_{p'}(Q))'}\ .
\end{equation*}
Denoting by $C([t_1, t_2]; X)$, $C^\alpha([t_1, t_2]; X)$ and $L^q(t_1, t_2; X)$ the usual spaces of continuous, H\"older and Lebesgue functions respectively with values in a Banach space $X$, we have the following isomorphisms: $W^{1,0}_2(Q) \simeq L^2((t_1, t_2); W^{1,2}(\T))$, and
\begin{multline*}
\H_2^1(Q) \simeq \{u \in L^2(t_1, t_2; W^{1,2}(\T)), \, \partial_t u \in  (\, L^2(t_1, t_2; W^{1,2}(\T))\, )'  \} \\ \simeq \big\{u \in L^2(t_1, t_2; W^{1,2}(\T)), \, \partial_t u \in  L^2\big(t_1, t_2; (W^{1,2}(\T))'\big) \big \},
\end{multline*}
and the latter is known to be continuously embedded into $C([t_1, t_2]; L^2(\T))$ (see, e.g., \cite[Theorem XVIII.2.1]{DL}). Sometimes, we will use the compact notation $C(X)$ and $L^q(X)$.

\subsection{Sobolev regularity of solutions to Fokker-Planck equations}\label{regadj}

In this section we collect a few existence and regularity properties of solutions to
\begin{equation}\label{adj}
\begin{cases}
-\partial_t\rho(x,t)-\Delta \rho(x,t)+\mathrm{div}(b(x,t)\rho(x,t))=0&\text{ in }Q_\tau\ ,\\
\rho(x,\tau)=\rho_\tau(x)&\text{ in }\T
\end{cases}
\end{equation} 
Here, $\tau \in (0,T]$, $Q_\tau:=\T\times (0,\tau)$ and $Q_{s,\tau}:= \T \times (s,\tau)$. We will assume that
\begin{equation}\label{rhoass}
\rho_\tau\in L^\infty(\T), \quad \rho_\tau\geq0 \ \text{ a.e.} 
\end{equation}

We first recall that such a transport equation with diffusion is well-posed in an energy space, and has bounded solutions, when $|b|\in L^\sP(Q_\tau)$ with $\sP$ satisfying the so-called Aronson-Serrin condition $\sP \ge d+2$. For a proof of the following classical result, see e.g. \cite[Proposition 2.3]{CG2}, or \cite{BCCS, BOP, LSU}.
\begin{prop}\label{wfkpp}
Let $|b|\in L^\sP(Q_\tau)$, with $\sP \ge d+2$, and $\rho_\tau$ be as in \eqref{rhoass}. Then, there exists a unique weak solution $\rho\in\H_2^1(Q_\tau) \cap L^\infty((0,\tau); L^r(\T))$ for all $r \ge 1$ to \eqref{adj}, i.e.
\begin{equation}\label{wfkp}
 -\int_s^\tau \langle \partial_t\rho(t), \varphi(t) \rangle dt + \iint_{Q_{s,\tau}} D\rho\cdot D\varphi  - b \rho \cdot D\varphi \,dxdt = 0
\end{equation}
for all $s > 0$ and $\varphi \in \H_2^1(Q_{s,\tau})$, and $\rho(\tau) = \rho_\tau$ in the $L^2$-sense. Moreover, $\rho$ is a.e. nonnegative on $Q_\tau$.  
\end{prop}

Our analysis will be based on regularity properties of $\rho$ in Sobolev spaces that depend on the integrability of $|b|$ against $\rho$ itself. Similar results already appeared in \cite{CT,BKRS,MPR, Porr}.

\begin{prop}\label{new2}
Let $\rho$ be the nonnegative weak solution to \eqref{adj} and $1<\sigma'<d+2$. Then, there exists $C>0$, depending on $T,\sigma',d$, such that
\begin{equation}\label{ineqint}
\|\rho\|_{\H^1_{\sigma'}(Q_\tau)}\leq C\left(\iint_{Q_\tau}|b(x,t)|^{m'}\rho\,dxdt+\|\rho_\tau\|_{L^{p'}(\T)}\right)
\end{equation}
where $m'=1+\frac{d+2}{\sigma}$ and $p' = \frac{d \sigma}{\sigma(d+1)-(d+2)}$ if $\sigma'>\frac{d+2}{d+1}$, while $p' = 1$ if $\sigma'<\frac{d+2}{d+1}$.
\end{prop}

As we will se below, if $\sigma'>\frac{d+2}{d+1}$ we obtain also
\begin{equation}\label{ineqint3}
\|\rho(t)\|_{L^{p'}(\T)} \leq C\left(\iint_{Q_\tau}|b(x,t)|^{m'}\rho\,dxdt+\|\rho_\tau\|_{L^{p'}(\T)}\right) \quad \text{for all $t \in [0,\tau]$}.
\end{equation}

\begin{proof}
The case $\sigma'<\frac{d+2}{d+1}$ is covered by \cite[Proposition 2.5]{CG2}, which is based on duality combined with maximal regularity arguments (following \cite{MPR}).
%The proof follows closely that of elliptic problems, cf \cite{MPRell}. 
We focus here on the case $\sigma'>\frac{d+2}{d+1}$, and prove the theorem via a (standard) method from weak solutions, that does not exploit parabolic Cald\`eron-Zygmund regularity. This is possible because $\frac{d+2}{d+1} < \sigma' < d+2$ implies $2 < m' < d+2$. \\
Set $\beta:=\frac{m'-2}{d+2-m'}$. To simplify, we use $\varphi:=\rho^\beta$ as a test function in the weak formulation of \eqref{adj} integrating on $Q_{t,\tau}:=\T\times(t,\tau)$, while the argument can be made rigorous by testing against $\varphi:=(\rho+\epsilon)^\beta$, and then letting $\epsilon\to0$. We have
\begin{multline*}
\frac{1}{\beta+1}\int_\T \rho(x,t)^{\beta+1}\,dx+\beta\iint_{Q_{t,\tau}}\rho^{\beta-1}|D\rho|^2\,dxdt\leq \iint_{Q_{t,\tau}}|b|\rho^\beta|D\rho|\,dxdt+\frac{1}{\beta+1}\int_{\T}|\rho_\tau|^{\beta+1}\,dx\\
\leq C_\beta\iint_{Q_{t,\tau}}|b|^2\rho^{\beta+1}\,dxdt+\frac{\beta}{4}\iint_{Q_{t,\tau}}\rho^{\beta-1}|D\rho|^2\,dxdt+\frac{1}{\beta+1}\int_{\T}|\rho_\tau|^{\beta+1}\,dx
\end{multline*}
In particular, noting that $D\rho^\frac{\beta+1}{2}=\frac{\beta+1}{2}\rho^{\frac{\beta-1}{2}}D\rho$, we write
\[
\beta\iint_{Q_{t,\tau}}\rho^{\beta-1}|D\rho|^2\,dxdt=\frac{4\beta}{(\beta+1)^2}\iint_{Q_{s,\tau}}|D\rho^{\frac{\beta+1}{2}}|^2\,dxdt
\]
to get that
\[
\frac{1}{\beta+1}\int_\T \rho(x,t)^{\beta+1}\,dx+\beta\iint_{Q_{t,\tau}}\rho^{\beta-1}|D\rho|^2\,dxdt\geq c_\beta \left[\int_\T (|\rho(x,t)|^{\frac{\beta+1}{2}})^2\,dx+\iint_{Q_{t,\tau}}|D\rho^{\frac{\beta+1}{2}}|^2\,dxdt\right].
\]
We first pass to the supremum over $t\in(0,\tau)$ and, by means of \cite[Proposition I.3.1]{DiBen} we deduce
\[
c_\beta\left[\mathrm{ess\ sup}_{t\in(0,\tau)}\int_\T (|\rho(x,t)|^{\frac{\beta+1}{2}})^2\,dx+\iint_{Q_{\tau}}|D\rho^{\frac{\beta+1}{2}}|^2\,dxdt\right]\geq c_{\beta,d}\left(\iint_{Q_{\tau}}\rho^{(\beta+1)\frac{d+2}{d}}\,dxdt\right)^{1-\frac{2}{d+2}}\ .
\]
We then have
\begin{multline*}
c_{\beta,d}\left(\iint_{Q_{\tau}}\rho^{(\beta+1)\frac{d+2}{d}}\,dxdt\right)^{1-\frac{2}{d+2}}+\frac{\beta}{2}\iint_{Q_{\tau}}\rho^{\beta-1}|D\rho|^2\,dxdt\\
\leq c_1\left(\iint_{Q_{\tau}}|b|^{m'}\rho\,dxdt\right)^{\frac{2}{m'}}\left(\iint_{Q_{\tau}}\rho^{\beta\frac{m'}{m'-2}+1}\,dxdt\right)^{1-\frac{2}{m'}}+c_2\|\rho_\tau\|_{L^{\beta+1}(\T)}^{\beta+1}\ .
\end{multline*}
We then note that $\sigma<d+2$ implies $m'>2$ and
\[
(\beta+1)\frac{d+2}{d}=\beta\frac{m'}{m'-2}+1=\frac{d+2}{d+2-m'}\ .
\]
We apply Young's inequality to the first term on the right-hand side to conclude
\begin{multline*}
c_1\left(\iint_{Q_{\tau}}|b|^{m'}\rho\,dxdt\right)^{\frac{2}{m'}}\left(\iint_{Q_{\tau}}\rho^{\beta\frac{m'}{m'-2}+1}\,dxdt\right)^{1-\frac{2}{m'}}\\
\leq c_3\left(\iint_{Q_{\tau}}|b|^{m'}\rho\,dxdt\right)^\frac{d}{d+2-m'}+\frac{c_{\beta,d}}{2}\left(\iint_{Q_{\tau}}\rho^{(\beta+1)\frac{d+2}{d}}\,dxdt\right)^{\frac{d}{d+2}}\ .
\end{multline*}
Hence
\begin{multline*}
c_{\beta,d}\left(\iint_{Q_{\tau}}\rho^{\frac{d+2}{d+2-m'}}\,dxdt\right)^{\frac{d}{d+2}}+\frac{\beta}{2}\iint_{Q_{\tau}}\rho^{\beta-1}|D\rho|^2\,dxdt\\
\leq c_4\left[\left(\iint_{Q_{\tau}}|b|^{m'}\rho\,dxdt\right)^\frac{d}{d+2-m'}+\|\rho_\tau\|_{L^{\frac{d}{d+2-m'}}(\T)}^{\frac{d}{d+2-m'}}\right]\\
\leq c_5\left[\iint_{Q_{\tau}}|b|^{m'}\rho\,dxdt+\|\rho_\tau\|_{L^{\frac{d}{d+2-m'}}(\T)}\right]^\frac{d}{d+2-m'}\ .
\end{multline*}
We thus conclude the estimate
\[
\|\rho\|_{L^{\frac{d+2}{d+2-m'}}(Q_{\tau})}\leq C\left(\iint_{Q_{\tau}}|b|^{m'}\rho\,dxdt+\|\rho_\tau\|_{L^{\frac{d}{d+2-m'}}(\T)}\right)
\]
and also
\begin{equation}\label{inter}
\iint_{Q_{\tau}}\rho^{\beta-1}|D\rho|^2\,dxdt\leq \bar C\left[\iint_{Q_{\tau}}|b|^{m'}\rho\,dxdt+\|\rho_\tau\|_{L^{\frac{d}{d+2-m'}}(\T)}\right]^\frac{d}{d+2-m'}\ .
\end{equation}
Finally, recalling that $\sigma'=\frac{d+2}{d+3-m'}$, we get the Sobolev estimate applying H\"older's inequality, using \eqref{inter} and finally exploiting Young's inequality as
\begin{multline*}
\|D\rho\|_{L^{\sigma'}(Q_{\tau})}\leq \|\rho^{(\beta-1)/2}D\rho\|_{L^2(Q_{\tau})}\|\rho^{(1-\beta)/2}\|_{L^{\frac{2(d+2)}{d+4-2m'}}(Q_{\tau})}\\
=\left(\iint_{Q_{\tau}}\rho^{\beta-1}|D\rho|^2\,dxdt\right)^{\frac12}\left(\iint_{Q_{\tau}}\rho^{\frac{d+2}{d+2-m'}}\,dxdt\right)^{\frac{d+4-2m'}{2(d+2)}}\\ \leq c_6\left(\iint_{Q_{\tau}}|b|^{m'}\rho\,dxdt+\|\rho_\tau\|_{L^{\frac{d}{d+2-m'}}(\T)}\right)\ .
\end{multline*}
The estimate on the time derivative in $(W^{1,0}_\sigma(Q_\tau))'$ can be obtained by duality, as in \cite[Proposition 2.4]{CG2}.
% The key step is a maximal regularity type result of Proposition \ref{propnew}, stated below. From \eqref{mreadj}, writing $b \rho = b \rho^{1/m'} \rho^{1/m}$, one uses generalized H\"older inequality (with exponents $\frac1{\sigma'} = \frac1{m'} + \frac1{m(\sigma')^*}$, where $\frac1{(\sigma')^*} = \frac1{\sigma'}-\frac1{d+2}$) and Young inequality to get
%\[
%\|\rho\|_{\H^1_{\sigma'}(Q_\tau)}\leq C\left(\frac1\eps\iint_{Q_\tau}|b(x,t)|^{m'}\rho\,dxdt+ \eps \|\rho\|_{L^{(\sigma')^*}(Q_\tau)} + \|\rho_\tau\|_{L^{p'}(\T)}\right).
%\]
%Then, thanks to the continuous embedding of $\H^1_{\sigma'}(Q_\tau)$ into $L^{(\sigma')^*}(Q_\tau)$, one concludes by choosing $\eps$ small enough.
\end{proof}

\begin{cor}\label{corrho}
Let $\rho$ be the nonnegative weak solution to \eqref{adj}. Then, there exists $C_1>0$, depending on $T,q,d$, such that if $ q < \frac{d+2}2$
\begin{equation*}
\sup_{t \in [0, \tau]} \|\rho(t)\|_{L^{p'}(\T)} + \|\rho\|_{L^{q'}(Q_\tau)}\leq C_1\left(\iint_{Q_\tau}|b(x,t)|^{\frac{d+2}q}\rho\,dxdt+\|\rho_\tau\|_{L^{p'}(\T)}\right)
\end{equation*}
where $p = \frac{d q}{(d+2)-2q}$, while if $ q > \frac{d+2}2$,
\begin{equation*}
\sup_{t \in [0, \tau]} \|\rho(t)\|_{W^{\frac{d}{d+2}-\frac2{q'}, \frac{(d+2)q'}{d+2+q'}}(\T)} + \|\rho\|_{L^{q'}(Q_\tau)}\leq C_1\left(\iint_{Q_\tau}|b(x,t)|^{\frac{d+2}q}\rho\,dxdt+\|\rho_\tau\|_{L^{1}(\T)}\right) .
\end{equation*}
\end{cor}

\begin{proof} The first estimate follows by Proposition \ref{new2}, applied with $m'=\frac{d+2}{q}$ (see also \eqref{ineqint3}), and the continuous embedding of $\H^1_{\sigma'}(Q_\tau)$ into $L^{q'}(Q_\tau)$. The second estimate follows analogously, using also the embedding of $\H^1_{\sigma'}(Q_\tau)$ into $C([0, \tau]; W^{\frac{d}{d+2}-\frac2{q'}, \frac{(d+2)q'}{d+2+q'}}(\T))$.
 \end{proof}

\section{Bounds of solutions to HJ equations in Lebesgue and H\"older spaces}\label{firstbounds}

We derive in this section some preliminary bounds on solutions to Hamilton-Jacobi equations via duality methods. Since we are in the setting of maximal regularity and $f \in L^q(Q_T)$, we assume that $u \in W^{2,1}_q(Q_T) \cap W^{1,0}_{\gamma q}$ is a strong solution to \eqref{hjb}, i.e. it solves the Hamilton-Jacobi equation almost everywhere. The initial datum $u_0$ is achieved in $W^{2-2/q,q}(\T)$, as suggested by the Lions-Peetre trace method in interpolation theory (cf Lemma \ref{embprel} below).  Note that
by classical embedding properties for Sobolev-Slobodeckij spaces, we have the following inclusions
\begin{equation}\label{initdat}
W^{2-\frac{2}{q},q}(\T)\hookrightarrow \begin{cases}
C^{2-\frac{d+2}{q}}(\T)&\text{ for }q>\frac{d+2}{2}\ ,\\
L^p(\T)&\text{ for }p\in[1,\infty)\text{ and }q=\frac{d+2}{2}\ ,\\
L^{\frac{dq}{d+2-2q}}(\T)&\text{ for } q<\frac{d+2}{2}\ .\\
\end{cases}
\end{equation}
Since we will work under the assumptions
\[
q \ge \frac{d+2}{\gamma'}, \qquad \gamma > 1 + \frac{2}{d+2},
\]
emebeddings of $W^{2,1}_q(Q_T)$ imply that $u$ is bounded in $W^{1,0}_{\gamma q}(Q_T)$ (so we can drop $u \in  W^{1,0}_{\gamma q}(Q_T)$ in the statements of our results), and $u \in L^2(0,T; W^{1,2}(\T))$. Furthermore,
\[
|Du|^{\gamma-1} \in L^{\sP}(Q_T) \qquad \text{for some $\sP \ge d+2$,}
\]
so the dual equation \eqref{adj} is well-posed in $\H_2^1$ whenever $b(x,t) = -D_pH(x, Du(x,t))$. Note finally that when $q > (d+2)/2$, any solution is automatically continuous and $u$ solves \eqref{hjb} in the weak sense used in \cite{CG2}. This always happens in the superquadratic regime $\gamma > 2$. On the other hand, in the subquadratic case $\gamma < 2$, $u$ is not necessarily continuous when $ (d+2)/\gamma' \le q < (d+2)/2$.

\begin{rem} The assumption
\[
\gamma > 1 + \frac{2}{d+2}
\]
guarantees that $u \in L^2(0,T; W^{1,2}(\T))$, so $u$ has finite energy and it can be safely used as a test function for the dual equation \eqref{adj}. One can drop this requirement in all of the following statements, relaxing to $$\gamma > 1 + \frac1{d+1},$$ so that $q >  \frac{d+2}{\gamma'} > 1$, and assuming {\it a priori} that $u \in L^2(0,T; W^{1,2}(\T))$ (which is always true for example when $u$ is a classical solution, as in Section \ref{mfgs} on MFG). One could also invoke methods from renormalized solutions, to deal also with the case $\gamma \le 1 + \frac1{d+1}$ but this is beyond the scopes of this paper.
\end{rem}

\subsection{Bounds in Lebesgue spaces}

We obtain in this section bounds on the $L^p$-norm of $u$, elaborating in particular the case $\frac{d+2}{\gamma'} \le q < \frac{d+2}2$, that is when $\gamma < 2$. 

$T_k(s) = \max\big\{-s,\min\{s,k\}\big\}$ below will denote the truncation operator at level $k > 0$, $u^+ = \max\{u,0\}$ and $u^- = (-u)^+$ the positive and negative part of $u$ respectively.

We start with some bounds on $u^+$, that are obtained with no restrictions on  $q \ge 1$.

\begin{lemma}\label{pboundP+} Assume that $H$ is non-negative. Let $u$ be a strong solution to \eqref{hjb} in $W^{2,1}_q(Q_T)$, $q \ge1$. There exists a positive constant ${C}_0$ (depending on $q, d, T$) such that
\begin{equation}\label{upiu}
\| u^+(\tau) \|_{L^p(\T)} \le \|u^+_0\|_{L^p(\T)}+C_0 \|f^+\|_{L^q(Q_\tau)},
\end{equation}
where $p = \frac{d q}{(d+2)-2q}$ if $q < \frac{d+2}2$, while $p = \infty$ if $q > \frac{d+2}2$.
\end{lemma}
%For a more detailed dependance of $C$ with respect to $u_0$ and $f$, see equations \eqref{upiu} and \eqref{umeno} below. 
Note that $W^{2,1}_q(Q_T)$ is embedded into $C([0,T]; L^p(\T))$ ($p$ as in the previous statement) so \eqref{upiu} has the form of an a priori bound.

\begin{proof}  We detail the proof in the case $q < \frac{d+2}2$ only. The case $q > \frac{d+2}2$ can be treated in an analogous way (essentially following \cite[Section 3]{CG2}). 

For $k > 0$, let $\mu = \mu_k$ be the weak non-negative solution of the following backward problem
\[
\begin{cases}
-\partial_t\mu(x,t)-\Delta \mu(x,t)=0&\text{ in }Q_\tau\ ,\\
\mu(x,\tau)=\frac{\big[T_k\big(u^+(x, \tau)\big)\big]^{p-1}}{\|u^+(\tau)\|_{p}^{p-1}}&\text{ in }\T
\end{cases}
\]
Note that $\|\mu_\tau\|_{L^{p'}(\T)}\le1$. Since $\mu$ solves an equation of the form \eqref{adj} with $b\equiv0$, by Corollary \ref{corrho} we have
\[
 \|\mu\|_{L^{q'}(Q_\tau)}\leq C,
\]
where $C$ does not depend on $k$. Since $u^+$ is a weak subsolution to
\[
\partial_t u^+(x,t)-\Delta u^+(x,t)\le \big[f(x,t) - H(x,Du(x,t))\big] \chi_{\{u > 0\}} \qquad \text{ in }Q_\tau,
\]
testing against $\mu$, and testing the equation for $\mu$ against $u$ we get that
\[
\int_\T u^+(\tau)\mu(\tau)\,dx \le \int_\T u^+_0\mu(0)\,dx+\iint_{Q_\tau \cap\{u> 0\}}f\mu\,dxdt-\iint_{Q_\tau\cap\{u> 0\}}H(x,Du)\mu \,dxdt.
\]
We apply H\"older's inequality to the second term of the right-hand side of the above inequality, the assumption $H\geq0$ on the Hamiltonian, and the fact that the backward heat equation preserves the $L^{p'}$ norm, i.e. $\|\mu(t)\|_{L^{p'}(\T)}\leq 1$ for all $t\in[0,\tau]$, to get, after sending $k \to \infty$, the desired inequality.
\end{proof}

We now proceed with some more delicate bounds on $u^-$, under the restriction $q > \frac{d+2}{\gamma'}$. We will use the following property of $H$: under the standing assumptions on $H$, the Lagrangian $L(x, v) = \sup_{p }\{\nu \cdot p - H(x, p)\}$ satisfies for some $C_L > 0$ (depending on $C_H$)
\begin{equation}\label{L1} 
C_L^{-1}|\nu|^{\gamma'} - C_L \le L(x, \nu) \le C_L |\nu|^{\gamma'}+ C_L.
\end{equation}

\begin{lemma}\label{pboundP-} Assume that \eqref{H} holds. Let $u$ be a strong solution to \eqref{hjb} in $W^{2,1}_q(Q_T)$, $q > \frac{d+2}{\gamma'}$. There exists a positive constant ${C}$ (depending on $C_H, q, d, T$) such that
\begin{multline}\label{umeno}
\| u^-(\tau) \|_{L^p(\T)} \le  \\
C\left( \| u^-_0 \|_{L^{p'}(\T)}+ \|f^-\|_{L^{q}(Q_\tau)}\right) + 2CC_L\tau \left( \| u^-_0 \|_{L^{p'}(\T)}+ \|f^-\|_{L^{q}(Q_\tau)}\right)^{\frac{q\gamma'}{q\gamma'-(d+2)} }  + C_L\tau.
\end{multline}
where $p = \frac{d q}{(d+2)-2q}$ if $q < \frac{d+2}2$, while $p = \infty$ if $q > \frac{d+2}2$.
\end{lemma}

\begin{proof} As before, we detail the case $q < \frac{d+2}2$ only. 
For $k > 0$, let $\rho = \rho_k$ be the weak non-negative solution of 
\[
\begin{cases}
-\partial_t\rho(x,t)-\Delta \rho(x,t)+\mathrm{div}\big(D_pH(x,Du(x,t))\chi_{\{u < 0\}}\rho(x,t)\big)=0&\text{ in }Q_\tau\ ,\\
\rho(x,\tau)=\frac{\big[T_k\big(u^-(x, \tau)\big)\big]^{p-1}}{\|u^-(\tau)\|_{p}^{p-1}}&\text{ in }\T
\end{cases}
\]
As before, $\|\mu_\tau\|_{L^{p'}(\T)}\le1$. By Corollary \ref{corrho} we have
\begin{equation}\label{equz}
 \|\rho(0)\|_{L^{p'}(\T)} + \|\rho\|_{L^{q'}(Q_\tau)}\leq C_1\left(\iint_{Q_\tau}|D_pH(x,Du)|^{\frac{d+2}q} \chi_{\{u < 0\}} \rho\,dxdt+\|\rho_\tau\|_{L^{p'}(\T)}\right),
\end{equation}
where $C$ does not depend on $k$. Since $u^-$ is a weak subsolution to
\[
\partial_t u^-(x,t)-\Delta u^-(x,t)\le \big[-f(x,t) + H(x,Du(x,t))\big] \chi_{\{u < 0\}} \qquad \text{ in }Q_\tau,
\]
testing against $\rho$, and testing the equation for $\rho$ against $u^-$ we get that
\begin{multline*}
\int_\T u^-(\tau)\rho(\tau)\,dx + \iint_{Q_\tau}\big[-D_pH(x,Du)\cdot Du^- - H(x,Du)\big]  \chi_{\{u < 0\}}  \rho \,dxdt \\ \le \int_\T u^-_0\rho(0)\,dx-\iint_{Q_\tau\cap\{u< 0\}}f\rho\,dxdt.
\end{multline*}
On one hand, in view of \eqref{L1} note that 
\begin{multline*}
\big[-D_pH(x,Du)\cdot Du^- - H(x,Du)\big]  \chi_{\{u < 0\}} =L(D_pH(x,-Du^-))\chi_{\{u < 0\}} \\ \ge \big[C_L^{-1}|D_pH(x,Du)|^{\gamma'} - C_L\big]\chi_{\{u < 0\}},
\end{multline*}
and on the other hand by H\"older's inequality
\begin{multline}\label{equz2}
\int_\T u^-(\tau)\rho(\tau)\,dx + C_L^{-1} \iint_{Q_\tau}|D_pH(x,Du)|^{\gamma'}  \chi_{\{u < 0\}}  \rho \,dxdt - C_L \iint_{Q_\tau} \chi_{\{u < 0\}} \rho \,dxdt \\ \le \| u^-_0 \|_{L^{p}(\T)} \|\rho(0)\|_{L^{p'}(\T)} + \|f^-\|_{L^{q}(Q_\tau)} \|\rho\|_{L^{q'}(Q_\tau)}   .
\end{multline}
%First one can verify that
%\[
%2- \frac2q - \frac d q = \frac2{q'}-\frac{d}{d+2} - d\frac{q'(d+1)-(d+2)}{q'(d+2)},
%\]
%so $\| u^-_0 \|_{W^{\frac2{q'}-\frac{d}{d+2}, \frac{q'(d+1)-(d+2)}{(d+2)q'}}(\T)} \le c \| u^-_0 \|_{W^{2-\frac2q, q}(\T)}$. 
Then, plugging \eqref{equz} into \eqref{equz2} we obtain
\begin{multline}\label{crucstep}
\int_\T u^-(\tau)\rho(\tau)\,dx + C_L^{-1} \iint_{Q_\tau}|D_pH(x,Du)|^{\gamma'}  \chi_{\{u < 0\}}  \rho \,dxdt \le \\
C_1\left( \| u^-_0 \|_{L^{p}(\T)}+ \|f^-\|_{L^{q}(Q_\tau)}\right) \left(\iint_{Q_\tau}|D_pH(x,Du)|^{\frac{d+2}q} \chi_{\{u < 0\}} \rho\,dxdt+1\right) + C_L \iint_{Q_\tau} \chi_{\{u < 0\}} \rho \,dxdt.
\end{multline}
Since $q > \frac{d+2}{\gamma'}$ one can use Young's inequality, and the fact that $\int \rho(t) dx \le \|\rho(t)\|_{L^{p'}(\T)}\le 1$ to get the desired inequality (after passing to the limit $k \to \infty$).
\end{proof}

Combining \eqref{upiu} and \eqref{umeno}, we get the following estimate in the case $q > \frac{d+2}{\gamma'}$.
\begin{cor}\label{pboundP} Assume that \eqref{H} holds. Let $u$ be a strong solution to \eqref{hjb} in $W^{2,1}_q(Q_T)$, $q > \frac{d+2}{\gamma'}$, and assume that
\[
\|u_0\|_{L^p(\T)}, \, \|f\|_{L^{q}(Q_T)} \le K.
\]
Then, there exists a positive constant ${C}$ (depending on $K, C_H, q, d, T$) such that
\begin{equation}\label{pbound}
\sup_{t \in [0,T]}\|u(t)\|_{L^p(\T)} \le C,
\end{equation}
where $p = \frac{d q}{(d+2)-2q}$ if $q < \frac{d+2}2$, while $p = \infty$ if $q > \frac{d+2}2$.
\end{cor}

Note that the previous result does not cover the critical case $q = \frac{d+2}{\gamma'}$. The rest of the section is devoted to this endpoint situation, and more precise information on the stability of solutions in $L^p$ will be obtained. This will be crucial in the subsequent analysis of maximal regularity. It is worth noting that constants appearing in estimates below will not depend just on the norms $\|f\|_{L^q}$, $\|u_0\|_{L^p}$,  but on finer properties of $f$ in $L^q$ and $u_0$ in $L^p$; see Remark \ref{cdep}.
% (see e.g. the discussion in \cite[Section 1.1]{Magliocca} and references therein).

Let $\Gamma(x,t)$ be the fundamental solution of the heat equation on $\T \times (0, \infty)$. Consider, for $k > 0$, the solution $u_k$ to the problem with truncated / regularized data, i.e.
\begin{equation}\label{hjbt}
\begin{cases}
\partial_t u_k(x,t)-\Delta u_k(x,t)+H(x,Du_k(x,t))=T_k\big(f(x,t)\big)&\text{ in }Q_T = \T \times (0,T)\ ,\\
u_k(x,0)=u_0(\cdot) \star \Gamma(\cdot, 1/k) \, (x) &\text{ in }\T.
\end{cases}
\end{equation}
In other words, $u_k$ solves an Hamilton-Jacobi equation with $L^\infty$ right-hand side and initial datum in $C^\infty$. Such an initial datum is actually $u_k(x,0) = z(x, 1/k)$, where $z$ is the solution to the heat equation
\[
\begin{cases}
\partial_t z(x,t)-\Delta z(x,t)=0&\text{ in }Q_T = \T \times (0,T)\ ,\\
z(x,0)=u_0(x) &\text{ in }\T,
\end{cases}
\]
and converges to $u_0$ in $W^{2-2/q,q}(\T)$ as $k\to \infty$. 
 The existence and uniqueness of a strong solution $u_k \in W^{2,1}_p(Q_T)$ for all $p \ge 1$ can be obtained using for example results in \cite{CG2}. We prove now estimates on $u-u_k$.

\begin{prop}\label{stability} Assume that \eqref{H} holds. Let $u$ and $u_k$ be strong solutions to \eqref{hjb} and \eqref{hjbt} respectively, $\gamma < 2$ and $q = \frac{d+2}{\gamma'}$. Then, there exists a positive constant ${C}$ (depending on $f, u_0, C_H, q, d, T$) such that
\begin{equation}\label{pboundc}
\sup_{t \in [0,T]}\|u(t) - u_k(t)\|_{L^p(\T)} \le C\big(\|f- T_k(f)\|_{L^q(Q_T)} + \|u_0- u_0\star \Gamma(1/k)\|_{L^p(\T)} \big),
\end{equation}
where $p = \frac{d q}{(d+2)-2q} = d \frac{\gamma-1}{2-\gamma}$.
\end{prop}

\begin{proof} Let $w = u-u_k$. Note that $w$  depends of course on $k$, but we will drop the subscript for simplicity. %Still, it will be crucial to obtain estimates on $w$ that do not depend on $k$. 

As before, we argue by duality, and estimate $w^+$ first. Fix $\tau \in (0,T]$ and $\rho = \rho_k$ be the weak nonnegative solution of 
\[
\begin{cases}
-\partial_t\rho(x,t)-\Delta \rho(x,t)+\mathrm{div}\big(D_pH(x,Du_k(x,t))\rho(x,t)\big)=0&\text{ in }Q_\tau\ ,\\
\rho(x,\tau)=\frac{\big[w^+(x, \tau)\big]^{p-1}}{\|w^+(\tau)\|_{p}^{p-1}}&\text{ in }\T.
\end{cases}
\]
Note that as in previous lemmas one should further truncate $\rho(\tau)$ to ensure the existence of $\rho$ in an energy space (cf Proposition \ref{wfkpp}), and then pass to the limit, but we will omit this step for brevity.

\textbf{Step 1:} bounds on $\iint|D_pH(x,Du_k)|^{\gamma'}   \rho$. Testing  \eqref{hjbt} against $\rho$, and testing the equation for $\rho$ against $u_k$ we get that
\begin{equation}\label{stabd}
 \iint_{Q_\tau}L\big(x, D_pH(x, Du_k)\big) \rho \,dxdt = -  \iint_{Q_\tau} T_k(f) \rho\,dxdt - \int_\T u_k(0)\rho(0)\,dx +  \int_\T u_k(\tau)\rho(\tau)\,dx.
\end{equation}
Recall that for all $t$, $\int_\T \rho(t) dx = \int_\T \rho(\tau) dx \le \|\rho(\tau)\|_{L^{p'}(\T)} =  1$. Then, for $h > 0$ that will be chosen below,
\begin{multline*}
-\iint_{Q_\tau} T_k(f) \rho\,dxdt  \le \iint_{Q_\tau} f^- \rho\,dxdt \le \iint_{Q_\tau \cap \{f^- \ge h\}} f^- \rho\,dxdt + h \tau \\ \le \|f^- \chi_{\{f^- \ge h\}}\|_{L^q(Q_\tau)} \| \rho \|_{L^{q'}(Q_\tau)} + h T.
\end{multline*}
Similarly, for any  $h_0 > 0$, by Young's inequality for convolutions and H\"older inequality,
\begin{multline*}
- \int_\T u_k(0)\rho(0)\,dx = - \int_\T u_0 \, \Gamma(1/k) \star \rho(0) \\ \le \int_{\T\cap \{u_0^- \ge h_0\}} u_0^- \, \Gamma(1/k) \star \rho(0) \,dx + h_0 \int_\T \Gamma(1/k) \star \rho(0) dx\\ \le  \|u_0^- \chi_{\{u_0^- \ge h_0\}}\|_{L^p(\T)} \| \Gamma(1/k) \star \rho(0)  \|_{L^{p'}(\T)} + h_0  \int_\T \Gamma(1/k) dx \int_\T \rho(0) dx \\
\le \|u_0^- \chi_{\{u_0^- \ge h_0\}}\|_{L^p(\T)} \|\rho(0)  \|_{L^{p'}(\T)} + h_0  .
\end{multline*}
Finally, applying Lemma \ref{pboundP+} to $u_k$ and noting that $(u_0\star \Gamma(1/k))^+ \le u_0^+ \star \Gamma(1/k)$ by the comparison principle,
\begin{multline*}
\int_\T u_k(\tau)\rho(\tau)\,dx \le \| u_k^+(\tau) \|_{L^p(\T)} \|\rho(\tau)\|_{L^{p'}(\T)} \le \|(u_0\star \Gamma(1/k))^+\|_{L^p(\T)}+C_0 \|T_k(f)^+\|_{L^q(Q_\tau)} \\ \le \|u_0^+\|_{L^p(\T)}+C_0 \|f^+\|_{L^q(Q_\tau)}
\end{multline*}
Pluggin the previous inequalities back in \eqref{stabd}, and using the bounds from below on $L$, we then get
\begin{multline*}
C_L^{-1} \iint_{Q_\tau}|D_pH(x,Du_k)|^{\gamma'}   \rho \,dxdt - C_L \iint_{Q_\tau} \rho \,dxdt \le h T +h_0 + \|u_0^+\|_{L^p(\T)}+C_0 \|f^+\|_{L^q(Q_\tau)} \\
+ \big(\|u_0^- \chi_{\{u_0^- \ge h_0\}}\|_{L^p(\T)} +  \|f^- \chi_{\{f^- \ge h\}}\|_{L^q(Q_\tau)} \big) \big(\| \rho(0) \|_{L^{p'}(\T)} + \| \rho \|_{L^{q'}(Q_\tau)} \big)
\end{multline*}
Then Corollary \ref{corrho} yields bounds on $\rho$, i.e.
\begin{multline*}
C_L^{-1} \iint_{Q_\tau}|D_pH(x,Du_k)|^{\gamma'}   \rho \,dxdt  \le (h +  C_L) T + h_0 + \|u_0^+\|_{L^p(\T)}+C_0 \|f^+\|_{L^q(Q_\tau)} \\
+ C_1\big(\|u_0^- \chi_{\{u_0^- \ge h_0\}}\|_{L^p(\T)} +  \|f^- \chi_{\{f^- \ge h\}}\|_{L^q(Q_\tau)} \big) \left( \iint_{Q_\tau}|D_pH(x,Du_k)|^{\gamma'}  \rho\,dxdt+1\right)
\end{multline*}
for some $C_1$ depending only on $T, d, \gamma$. Finally, $h$ and $h_0$ are chosen large enough so that $$\|u_0^- \chi_{\{u_0^- \ge h_0\}}\|_{L^p(\T)} +  \|f^- \chi_{\{f^- \ge h\}}\|_{L^q(Q_\tau)} \le \frac1{2 C_1 C_L} ,$$
that gives
\begin{equation}\label{crosse}
 \iint_{Q_\tau}|D_pH(x,Du_k)|^{\gamma'}   \rho \,dxdt  \le 2C_L[ (h +  C_L) T + h_0 + \|u_0^+\|_{L^p(\T)}+C_0 \|f^+\|_{L^q(Q_\tau)} ] +1 =: \overline C .
\end{equation}

\textbf{Step 2:} bounds on $w^+$. Taking the difference between \eqref{hjbt} and \eqref{hjb}, by convexity of $H(x, \cdot)$, $w^+$ is a weak subsolution of
\[
\partial_t w^+(x,t)-\Delta w^+(x,t)+D_pH(x,Du_k(x,t)) \cdot D w^+ \le \big[f- T_k\big(f(x,t)\big)\big] \chi_{\{w > 0\}} 
\]
Testing against $\rho$, and testing the equation for $\rho$ against $w^+$ gives
\[
\|w^+(\tau)\|_{L^p(\T)} = \int_{\T} w^+(\tau) \rho(\tau)\, dx \le \int_{\T} w^+(0) \rho(0)\, dx +  \iint_{Q_\tau} \big[f- T_k\big(f(x,t)\big)\big] \chi_{\{w > 0\}}  \rho\,dxdt .
\]
Hence, using Corollary \ref{corrho} and the estimate \eqref{crosse}  we obtain
\begin{multline*}
\|w^+\|_{L^p(\T)} \le C_1 \big( \|w^+(0)\|_{L^p(\T) } + \|f - T_k(f)\||_{L^q(\T) } \big) \left( \iint_{Q_\tau}|D_pH(x,Du_k)|^{\gamma'}  \rho\,dxdt+1\right)\\  \le 
C_1(\overline C+1) \big( \|u_0- u_0\star \Gamma(1/k)\|_{L^p(\T) } + \|f - T_k(f)\||_{L^q(\T) } \big) ,
\end{multline*}
which is ``half'' of the desired estimate.

\medskip

To get an analogous estimate for $\|w^-\|_{L^p(\T)}$, which allows to conclude since $w = u - u_k$, one can proceed as in Step 1 and 2, noting that $w^-$ satisfies
\[
\partial_t w^-(x,t)-\Delta w^-(x,t)+D_pH(x,Du(x,t)) \cdot D w^- \le \big[T_k\big(f(x,t)\big) - f\big] \chi_{\{w < 0\}} .
\]
To argue as before, it is sufficient to exploit properties of the dual problem
\[
\begin{cases}
-\partial_t\hat \rho(x,t)-\Delta \hat \rho(x,t)+\mathrm{div}\big(D_pH(x,Du(x,t))\hat \rho(x,t)\big)=0&\text{ in }Q_\tau\ ,\\
\hat \rho(x,\tau)=\frac{\big[w^-(x, \tau)\big]^{p-1}}{\|w^-(\tau)\|_{p}^{p-1}}&\text{ in }\T.
\end{cases}
\]
\end{proof}

\begin{rem}\label{cdep} Note that \eqref{pboundc} directly yields the estimate $$\sup_{t \in [0,T]}\|u(t)\|_{L^{d\frac{\gamma-1}{2-\gamma}}(\T)} \le C,$$ thus extending \eqref{pbound} up to the critical case $q = (d+2)/\gamma'$. Let us focus on the the way $C$ depends on $f$ and $u_0$. When $q > (d+2)/\gamma'$, $C$ depends on $\|f\|_{L^q}$ and $\|u_0\|_{L^p}$ only. At the endpoint $q = (d+2)/\gamma'$, $C$ above (and similarly the constant in \eqref{pboundc}) is proportional to $\overline C$ defined in \eqref{crosse}, that depends in turn on $ \|u_0^+\|_{L^p(\T)}, \|f^+\|_{L^q(Q_\tau)}, h, h_0$, where $h, h_0$ are such that
\[
\|u_0^- \chi_{\{u_0^- \ge h_0\}}\|_{L^p(\T)} +  \|f^- \chi_{\{f^- \ge h\}}\|_{L^q(Q_\tau)} \le \frac1{2 C_1 C_L} ,
\]
and $C_1 = C_1(T, q, d)$ is as in Corollary \ref{corrho}. Thus, these constants {\it remain bounded when $f$ and $u_0$ vary in bounded and equi-integrable sets in $L^q(Q_T)$ and $L^p(\T)$ respectively}. This is completely in line with results in \cite{Magliocca}.
\end{rem}

\subsection{Bounds in H\"older spaces}

We now proceed with bounds on $u$ in H\"older spaces, which will be obtained in particular when $\gamma \ge 2$, that is when $q$ is necessarily greater than $\frac{d+2}2$ (and therefore $u$ is continuous).

\begin{proof}[Proof of Theorem \ref{mainholder}]
\textbf{Step 1.} Since we have the representation $H(x,Du(x, t))=\sup_{\nu \in \R^d}\{\nu\cdot Du(x, t)-L(x,\nu)\}$ for a.e. $(x,t) \in Q_T$, we get
\begin{multline}\label{HJoptimalLOC}
\int_0^\tau \langle \partial_t u(t), \varphi(t) \rangle dt +  \iint_{Q_{\tau}}  \partial_i u(x, t) \, \partial_j( \varphi(x, t)) + [\Xi(x, t) \cdot Du(x, t)-L(x,\Xi(x, t))] \varphi \, dxdt \\
\le  \iint_{Q_{\tau}} f(x, t) \varphi(x, t) \,dxdt
\end{multline}
for any measurable $\Xi : Q_{\tau} \to \R^d$ such that $L(\cdot, \Xi(\cdot, \cdot)) \in L^r(Q_{\tau})$ and $\Xi\cdot Du \in  L^r(Q_{\tau})$, $r > 1$, and test function $\varphi \in  \H_2^1(Q_{\tau}) \cap L^{r'}(Q_{\tau})$. The previous inequality becomes an equality if $\Xi(x,t) = D_pH(x,Du(x,t))$ in $Q_{\tau}$.

Fix now any $\tau \in [0, T]$, and let $\bar x, \bar y \in \T$ be such that
\[
u(\bar y) - u(\bar x) = |\bar y - \bar x|^\alpha \cdot  [u(\cdot, \tau)]_{C^\alpha(\T)}.
\]
Let $\rho_\tau$ be any smooth non-negative function satisfying $\int_\T \rho_\tau = 1$, and $\rho \in  \H_2^1(Q_\tau) \cap L^{r'}(Q_{\tau})$ (for all $r' > 1$) be the solution to
\[
\begin{cases}
-\partial_t\rho(x,t)-\Delta \rho(x,t)+\mathrm{div}\big(D_pH(x,Du(x,t))\rho(x,t)\big)=0&\text{ in }Q_\tau\ ,\\
\rho(x,\tau)=\rho_\tau(x)&\text{ in }\T.
\end{cases}
\]
Use now \eqref{HJoptimalLOC} with $\Xi(x,t) = D_pH(x,Du(x,t))$ and $\varphi=\rho$, and $u \in \H_2^1(Q_T)$ as a test function for the equation satisfied by $\rho$ to get
\begin{multline}\label{optloc}
\int_{\T}u(x,\tau)\rho_{\tau}(x)dx= \int_{\T}u_0(x)\rho(x,0)dx + \iint_{Q_{\tau}}f(x,t)\rho(x,t)dxdt\\
+\iint_{Q_{\tau}}L\big(x,D_pH(x,Du(x,t))\big)\rho(x,t)dxdt.
\end{multline}
Setting $\xi = \bar y - \bar x$, one can easily check that $\hat{\rho}(x,t):=\rho(x-\xi,t)$ satisfies
\[
\begin{cases}
-\partial_t \hat \rho(x,t)-\Delta \hat\rho(x,t)+\mathrm{div}\big(D_pH(x-\xi,Du(x-\xi,t))\hat \rho(x,t)\big)=0&\text{ in }Q_\tau\ ,\\
\hat \rho(x,\tau)=\rho_\tau(x-\xi)&\text{ in }\T.
\end{cases}
\]
As before, plugging $\Xi(x,t) = D_pH(x-\xi,Du(x-\xi,t))$ and $\varphi =  \hat \rho$ into \eqref{HJoptimalLOC}, and using $u \in \H_2^1(Q_T)$ as a test function for the equation satisfied by $\hat \rho$ yields
\begin{multline*}
\int_{\T}u(x,\tau)\hat{\rho}_\tau(x)dx \le \int_{\T}u_0(x)\hat{\rho}(x,0)dx + \\
\iint_{Q_{\tau}} L(x,D_pH(x-\xi,Du(x-\xi,t))) \hat \rho \, dxdt+ \iint_{Q_{\tau}} f \hat \rho \,dxdt 
\end{multline*}
which, after the change of variables $x - \xi \mapsto x$, becomes
\begin{multline}\label{suboptloc}
\int_{\T}u(x+\xi,\tau)\rho_\tau(x)dx \le \int_{\T}u_0(x+\xi)\rho(x,0)dx  + \\
\iint_{Q_{\tau}} L(x+\xi,D_pH(x,Du(x,t))) \rho \, dxdt+ \iint_{Q_{\tau}} f \hat \rho \,dxdt .
\end{multline}
Taking the difference between \eqref{suboptloc} and \eqref{optloc} we obtain
\begin{multline}\label{step1}
\int_{\T}\big(u(x + \xi, \tau)-  u(x,\tau)\big) {\rho}_{\tau}(x)dx \le 
\int_{\T}\big(u_0(x + \xi)-  u_0(x)\big) \rho(x,0)dx + \\
+ \iint_{Q_{\tau}} \Big(L(x+\xi,D_pH(x,Du(x,t))) - L(x,D_pH(x,Du(x,t)))\Big)\rho(x,t) \, dxdt\\
+ \iint_{Q_{\tau}} f(x,t) \big(\rho(x-\xi, t)- \rho(x,t) \big) \,dxdt . 
\end{multline}

\textbf{Step 2.} To estimate the terms appearing in the right hand side of \eqref{step1}, we first derive bounds on $\rho$. We stress that constants $C, C_1, \ldots$ below are not going to depend on $\tau$ and $\rho_\tau$. Rearranging \eqref{optloc} we have
\begin{multline*}
\iint_{Q_{\tau}}L\big(x,D_pH(x,Du(x,t))\big)\rho(x,t)dxdt = \int_{\T}u(x,\tau)\rho_{\tau}(x)dx - \int_{\T}u_0(x)\rho(x,0)dx \\ - \iint_{Q_{\tau}}f(x,t)\rho(x,t)dxdt,
\end{multline*}
and by \eqref{L1} and bounds on $\|u\|_\infty$ of Proposition \ref{pboundP} we get
\[
C_L^{-1} \iint_{Q_\tau} |D_pH(x,Du(x,t))|^{\gamma'} \rho(x,t) dx dt \le C + \|f\|_{L^q(Q_\tau)}\|\rho\|_{L^{q'}(Q_\tau)}.
\]
Using Corollary \ref{corrho},
\begin{multline*}
C_L^{-1} \iint_{Q_\tau} |D_pH(x,Du(x,t))|^{\gamma'} \rho(x,t) dx dt \le C \\ 
+ C_1\|f\|_{L^q(Q_\tau)}\left(\iint_{Q_\tau}|D_pH(x,Du(x,t))|^{\frac{d+2}q}\rho\,dxdt+\|\rho_\tau\|_{L^{p'}(\T)}\right).
\end{multline*}
This provides a control on $\iint_{Q_\tau} |D_pH(Du)|^{\gamma'} \rho$, and by means of Proposition \ref{new2},
\begin{equation}\label{estik}
\iint_{Q_\tau} |D_pH(x,Du(x,t))|^{\gamma'} \rho(x,t) dx dt + \|\rho\|_{\H^1_{\frac{d+2}{d+3-\gamma'}}(Q_\tau)} \le C_2.
\end{equation}

\textbf{Step 3.}
First, recalling that $\int_{\T} \rho(0) = 1$,
\[
\int_{\T}\big(u_0(x + \xi)-  u_0(x)\big) \rho(x,0)dx \le |\xi|^\alpha [u_0]_{C^\alpha(\T)}
\]

As for the second term in \eqref{step1}, $L(x, v) = \sup_{p \in \R^d}\{v \cdot p-H(x,p)\}$, and if $v = D_pH(x,p)$, then $L(x, v) = \nu \cdot p - H(x, p)$, hence
\[
L(x+\xi,D_pH(x,Du(x,t))) - L(x,D_p H(x,Du(x,t)) \le H(x, Du(x,t)) - H(x+\xi, Du(x,t)).
\]
Next, using \eqref{Ha},
\begin{multline*}
\iint_{Q_{\tau}} \Big(L(x+\xi,D_pH(x,Du(x,t))) - L(x,D_pH(x,Du(x,t)))\Big)\rho(x,t) \, dxdt  \\
\le \iint_{Q_{\tau}} \Big( H(x, Du(x,t)) - H(x+\xi, Du(x,t)) \Big)\rho(x,t) \, dxdt  \\
\le C_H |\xi|^\alpha \iint_{Q_{\tau}} \Big( |D_pH(x,Du(x,t))|^{\gamma'}+1 \Big)\rho(x,t) \, dxdt.
\end{multline*}

Finally, we apply the embeddings of Propositions \ref{embnik} and \ref{emb1} (with $\delta=q'$, $p=\frac{d+2}{d+3-\gamma'}$ and hence $\alpha=\gamma'-\frac{d+2}{q}$) to get
\begin{multline*}
\left|  \iint_{Q_{\tau}} f(x,t) \big(\rho(x-\xi, t)- \rho(x,t) \big) \,dxdt \right |   \\
\le |\xi|^\alpha \iint_{Q_{\tau}} |f(x,t)|\,\frac{|\big(\rho(x-\xi, t)- \rho(x,t) \big)|}{|h|^\alpha} \,dxdt \le |\xi|^\alpha \|f\|_{L^{q}(Q_{\tau})} \|\rho\|_{L^{q'}(N^{\alpha,q'}(\T))}\\
\leq  C|\xi|^\alpha \|f\|_{L^{q}(Q_{\tau})} \|\rho\|_{L^{q'}(W^{\alpha,q'}(\T))}
\leq  |\xi|^\alpha \|f\|_{L^{q}(Q_{\tau})}\|\rho\|_{\H_{\frac{d+2}{d+3-\gamma'}}^1(Q_{\tau})}.
\end{multline*}
Plugging now all the estimates in \eqref{step1} and using \eqref{estik} we obtain
\begin{equation*}
 \int_{\T}\big(u(x + \xi, \tau)-  u(x,\tau)\big) {\rho}_{\tau}(x)dx \le C_1 |\xi|^\alpha.
\end{equation*}
It is now sufficient to recall that $\rho_\tau$ can be any smooth non-negative function satisfying $\int_\T \rho_\tau = 1$, so
\[
|\bar y - \bar x|^\alpha \cdot  [u(\cdot, \tau)]_{C^\alpha(\T)} = u(\bar y) - u(\bar x) \le C_1 |\bar y - \bar x|^\alpha
\]
and we have the assertion.
\end{proof}

\begin{rem}\label{holder2} Adding an additional time localization term in the previous procedure, as in \cite{CG2}, it is possible to obtain H\"older bounds that are independent of the initial datum, but just depend on the sup-norm of the solution. This indicates that the equation regularizes at H\"older scales, and weak solutions (in an appropriate sense) become instantaneously H\"older continuous at positive times. 
\end{rem}

\section{Maximal $L^q$-regularity}\label{sec;max}
We start with a straightforward consequence of parabolic regularity results for linear equations.
\begin{prop}
Assume that \eqref{H} holds. Let $u$ be a strong solution to \eqref{hjb} in $W^{2,1}_q(Q_T)$. Then,
\begin{equation}\label{pertheat}
\|u\|_{W^{2,1}_q(Q_T)}\leq C(\| Du \|_{L^{\gamma q}(Q_T)}^\gamma+\|f\|_{L^q(Q_T)}+\|u_0\|_{W^{2-\frac{2}{q},q}(\T)}+1)
\end{equation} 
for some positive constant $C$ depending on $q, d, C_H$.
\end{prop}

\begin{proof}
The proof is an easy consequence of well-known Cald\`eron-Zygmund type maximal regularity results for heat equations with potential
\[
\begin{cases}
\partial_t u(x,t)-\Delta u(x,t)=V(x,t)&\text{ in }Q_\tau\ ,\\
u(x,0)=u_0(x)&\text{ in }\T,
\end{cases}
\]
which satisfies the estimate (see \cite{Lamberton,LSU}, or \cite{HP} and the references therein)
\[
\|u\|_{W^{2,1}_q(Q_T)}\leq C(\|V\|_{L^q(Q_T)}+\|u_0\|_{W^{2-\frac{2}{q},q}(\T)}).
\]
To get \eqref{pertheat}, it is now sufficient to choose $V = -H(x,Du) + f$, and use the assumption \eqref{H}.
\end{proof}

We now proceed with our main result on maximal regularity for \eqref{hjb} when $q > (d+2)/\gamma'$.

\begin{proof}[Proof of Theorem \ref{maxreg1}] To prove the assertion, we will combine the estimates derived in Section \ref{firstbounds} with Gagliardo-Nirenberg type interpolation inequalities.

\smallskip

{\it The subquadratic case $\gamma < 2$. } We start from Proposition \ref{pboundP}, which gives
\begin{equation}\label{sboun}
\sup_{t \in [0,T]}\|u(t)\|_{L^s(\T)} \le C,
\end{equation}
for any $s \le p = \frac{d q}{(d+2)-2q}$ if $q < \frac{d+2}2$, while $s \le \infty$ if $q > \frac{d+2}2$. Recall then the classical Gagliardo-Nirenberg inequality (\cite[Lecture II Theorem p.125-126]{N})
\begin{equation}\label{gn1}
\|Du(t)\|_{L^{\gamma q}(\T)}\leq C_1\| u(t)\|^\theta_{W^{2,q}(\T)}\|u(t)\|_{L^s(\T)}^{1-\theta}
\end{equation}
for $s \in [1,\infty]$ and $\theta \in [1/2, 1)$ satisfying
\[
\frac1{\gamma q}=\frac{1}{d}+\theta\left(\frac1q-\frac{2}{d}\right)+(1-\theta)\frac1s\ .
\]
Note that since $q > \frac{d+2}{\gamma'}$, we have $p>\frac{d(\gamma-1)}{2-\gamma}$, and therefore it is possible to choose $s$ (close to $\frac{d(\gamma-1)}{2-\gamma}$) so that $\theta \in [1/2, 1/\gamma)$ and \eqref{sboun} and \eqref{gn1} holds. Then, raising \eqref{gn1} to $\gamma q$ and integrating on $(0,T)$ yields 
\[
\int_0^T \|Du(t)\|^{\gamma q}_{L^{\gamma q}(\T)} dt \leq C_1^{\gamma q}\left(\sup_{t \in [0,T]} \|u(t)\|_{L^s(\T)}^{1-\theta}\right)^{\gamma q}
 \int_0^T \| u(t)\|^{\gamma \theta q} _{W^{2,q}(\T)} dt ,
\]
and since $\gamma\theta < 1$,
\begin{equation}\label{GNg}
\|Du\|_{L^{\gamma q}(Q_T)}\leq C_2 \|u\|_{L^{q}(0,T;W^{2,q}(\T))}^\theta\|u\|_{L^\infty(0,T;L^{s}(\T))}^{1-\theta}.
\end{equation}
Plugging \eqref{sboun} and \eqref{GNg} into \eqref{pertheat} we obtain
\[
\|u\|_{W^{2,1}_q(Q_T)}
\leq C_3(\|u\|^{\theta \gamma}_{W^{2,1}_q(Q_T)}+\|f\|_{L^q(Q_T)}+\|u_0\|_{W^{2-\frac{2}{q},q}(\T)}),
\]
and we conclude the assertion because $\theta \gamma < 1$.
\smallskip

{\it The superquadratic case $\gamma \ge 2$. } We start from H\"older bounds of Theorem \ref{mainholder}, namely
\begin{equation}\label{abound}
\sup_{t \in [0,T]}\|u(t)\|_{C^\alpha(\T)} \le C_1,
\end{equation}
where $\alpha = \gamma'-\frac{d+2}{q}$ (or $\alpha \in (0,1)$ when $q \ge \frac{d+2}{\gamma'-1}$), and invoke the following Miranda-Nirenberg interpolation inequality (see \cite{NHolder,Miranda, Maugeri})
\[
\|Du(t)\|_{L^{\gamma q}(\T)}\leq C\|u(t)\|_{W^{2,q}(\T)}^\theta\|u(t)\|_{C^{\alpha}(\T)}^{1-\theta},
\]
where $\theta\in\left[\frac{1-\alpha}{2-\alpha},1\right)$ satisfies
\begin{equation*}\label{condGN1}
\frac{1}{\gamma q}=\frac{1}{d}+\theta\left(\frac{1}{q}-\frac{2}{d}\right)-(1-\theta)\frac{\alpha}{d}.
\end{equation*}
Choosing $\theta = \frac{1-\alpha}{2-\alpha}$ (or $\alpha$ close enough to $1$ when $q \ge \frac{d+2}{\gamma'-1}$), we have $\theta \gamma < 1$ if and only if 
\[
q>\frac{(d+2)(\gamma-1)}{2}.
\]
Hence,
\[
 \|Du\|_{L^{\gamma q}(Q_T)}\leq C C^{1-\theta}_1 \|u\|_{L^q(0,T;W^{2,q}(\T))}^\theta.
\]
Plugging this inequality into \eqref{pertheat} and using the fact that $\gamma \theta < 1$, we conclude.

\end{proof}

We now consider the maximal regularity problem in the limiting case $q = (d+2)/\gamma'$. The scheme of the proof is similar to the one of Theorem  \ref{maxreg1}, but requires an additional step involving solutions $u_k$ to the regularized problem \eqref{hjbt}, that is; for $k >0$,
\begin{equation}
\begin{cases}
\partial_t u_k(x,t)-\Delta u_k(x,t)+H(x,Du_k(x,t))=T_k\big(f(x,t)\big)&\text{ in }Q_T = \T \times (0,T)\ ,\\
u_k(x,0)=u_0(\cdot) \star \Gamma(\cdot, 1/k) \, (x) &\text{ in }\T.
\end{cases}
\end{equation}

\begin{proof}[Proof of Theorem \ref{maxregc}] Let $w = u - u_k$, $k$ to be chosen. From Proposition \ref{stability}, we have the existence of $C$ depending on $f, u_0, C_H, q, d, T$ such that
\[
\sup_{t \in [0,T]}\|w(t)\|_{L^p(\T)} \le C\big(\|f- T_k(f)\|_{L^q(Q_T)} + \|u_0- u_0\star \Gamma(1/k)\|_{L^p(\T)} \big), \quad p = d \frac{\gamma-1}{2-\gamma}.
\]
The Gagliardo-Nirenberg inequality reads
\[
\|Dw(t)\|_{L^{\gamma q}(\T)}\leq C_2 \| w(t)\|^{\frac1\gamma}_{W^{2,q}(\T)}\|w(t)\|_{L^p(\T)}^{1-\frac1\gamma},
\]
where $C_2$ depends on $d, p, q$. Thus,
\begin{equation}\label{eqrgj}
 \|Dw\|^\gamma_{L^{\gamma q}(Q_T)}\leq C_2^\gamma C^{\gamma-1}\big(\|f- T_k(f)\|_{L^q(Q_T)} + \|u_0- u_0\star \Gamma(1/k) \|_{L^p(\T)} \big)^{\gamma-1} \|w\|_{L^q(0,T;W^{2,q}(\T))}.
\end{equation}
Note now that $w$ solves a.e. on $Q_T$
\begin{equation}\label{lineq}
\partial_t w(x,t)-\Delta w(x,t)=H(x,Du_k(x,t)) - H(x,Du(x,t)) +f (x,t) - T_k\big(f(x,t)\big),
\end{equation}
and that, by assumptions on $H$, $|D_pH(x,p)| \le C_H' (|p|^{\gamma-1}+1)$, so by Young's inequality
\begin{align*}
|H(x,Du_k(x,t)) - H(x,Du(x,t))| & \le |Dw(x,t)| \cdot \max\{|D_pH(x,Du_k(x,t))|, |D_pH(x,Du(x,t))|\} \\
& \le C_3 (|Du_k(x,t)|^\gamma + |Du(x,t)|^\gamma+ |Dw(x,t)|^\gamma +1 ) \\
& \le C_4 (|Du_k(x,t)|^\gamma + |Dw(x,t)|^\gamma +1 )\ ,
\end{align*}
where $C_3, C_4$ depend on $C_H$ only.
Then, by maximal regularity applied to the linear equation \eqref{lineq},
\begin{multline*}
\|w\|_{L^q(0,T;W^{2,q}(\T))} \le \\ C_5 \Big( \| H(x,Du_k) - H(x,Du) \|_{L^q(Q_T)} + \|f - T_k\big(f\big) \|_{L^q(Q_T)} + \|u_0- u_0\star \Gamma(1/k)\|_{W^{2-\frac{2}{q},q}(\T)}\Big) \le \\
C_6\| Dw \|^\gamma_{L^{\gamma q}(Q_T)} + C_6 \Big(\| Du_k \|^\gamma_{L^{\gamma q}(Q_T)}+ \|f - T_k\big(f\big) \|_{L^q(Q_T)} + \|u_0- u_0\star \Gamma(1/k) \|_{W^{2-\frac{2}{q},q}(\T)} + 1 \Big) ,
\end{multline*}
where $C_6$ depends on $C_H ,d,q$. Plugging now this inequality into \eqref{eqrgj} yields 
\[
 \|Dw\|^\gamma_{L^{\gamma q}(Q_T)}\leq C_6 C_2^\gamma C^{\gamma-1}\big(\|f- T_k(f)\|_{L^q(Q_T)} + \|u_0- u_0\star \Gamma(1/k)\|_{L^p(\T)} \big)^{\gamma-1} \| Dw \|^\gamma_{L^{\gamma q}(Q_T)} + \cdots.
\]
Hence, we choose $\bar k$ large enough so that
\begin{equation}\label{kchoice}
C_6 C_2^\gamma C^{\gamma-1}\big(\|f- T_{\bar k}(f)\|_{L^q(Q_T)} + \|u_0- u_0\star \Gamma(1/\bar k)\|_{L^p(\T)} \big)^{\gamma-1} \le \frac12, 
\end{equation}
to get
\begin{multline}\label{stimadw}
\|Dw\|^\gamma_{L^{\gamma q}(Q_T)}\leq 
\| Du_{\bar k} \|^\gamma_{L^{\gamma q}(Q_T)} + \|f - T_{\bar k}\big(f\big) \|_{L^q(Q_T)} + \|u_0- u_0\star \Gamma(1/{\bar k})\|_{W^{2-\frac{2}{q},q}(\T)} + 1  \\
\leq \| Du_{\bar k} \|^\gamma_{L^{\gamma q}(Q_T)} + 2\|f \|_{L^q(Q_T)} + 2\|u_0\|_{W^{2-\frac{2}{q},q}(\T)} + 1 .
\end{multline}

Since $u_{\bar k}(0)$ is smooth and $T_{\bar k}(f) \in L^\infty(Q_T)$, we can apply Theorem \ref{maxreg1} to $u_{\bar k}$ solving \eqref{hjbt} to estimate $Du_{\bar k}$. Indeed, pick any $\bar q > q$. Then,
\[
\|T_{\bar k}(f)\|_{L^{\bar q}(Q_T)}+\|u_0\star \Gamma(1/{\bar k})\|_{W^{2-\frac{2}{\bar q},\bar q}(\T)} \le \bar k + C_5  {\bar k}^{ \left( \frac d2 + 1\right)\left( \frac 1q - \frac1{\bar q}\right)} \|u_0\|_{W^{2-\frac{2}{q},q}(\T)}
\]
in view of standard decay estimates for the heat equation ($C_5$ depends on $d, q ,\bar q$ only, see e.g. \cite[Chapter 15]{taylor}).
Therefore, by Theorem \ref{maxreg1},
\begin{equation}\label{stimaduk}
\| Du_{\bar k} \|^\gamma_{L^{\gamma \bar q}(Q_T)} \le C_{\bar k},
\end{equation}
where $C_{\bar k}$ depends on ${\bar k}, \|u_0\|_{W^{2-\frac{2}{q},q}(\T)}, q, d, C_H, T$. Actually, $Du$ can be proven to be bounded in $L^\infty(Q_T)$, see \cite{CG2}. It is now straightforward to conclude. Indeed, 
\[
\|Du\|_{L^{\gamma q}(Q_T)} \leq \|Dw\|_{L^{\gamma q}(Q_T)} + \|Du_{\bar k}\|_{L^{\gamma q}(Q_T)},
\]
and the assertion follows by \eqref{stimadw} and \eqref{stimaduk}.
\end{proof}
 
\begin{rem}\label{remlim} We claim that $C$ appearing in the statement of Theorem \ref{maxregc} remains bounded when
\begin{itemize}
\item[$\bullet$] \ $f$ varies in a bounded and {\it equi-integrable} set $\mathcal F \subset L^q(Q_T)$, and
\item[$\bullet$] \ $u_0$ varies in a bounded set  $\mathcal U_0 \subset W^{2-2/q,q}(\T)$.
\end{itemize}
Indeed, in addition to $\|u_0\|_{W^{2-\frac{2}{q},q}(\T)}, q, d, T, C_H$, the constant $C$ crucially depends on $\bar k$ appearing in \eqref{kchoice}. This is chosen in the proof large enough so that
\[
\big(\|f- T_{\bar k}(f)\|_{L^q(Q_T)} + \|u_0- u_0\star \Gamma(1/\bar k)\|_{L^p(\T)} \big) \le c.
\]
In turn, $c = (2 C_6 C_2^\gamma C^{\gamma-1})^{-1}$ is independent of $f \in \mathcal F$ and $u_0 \in \mathcal U_0$, since they vary in bounded and equi-integrable sets in $L^q(Q_T)$ and $L^p(\T)$ respectively, cf. Remark \ref{cdep}. Note that by Sobolev embeddings, the closure of $\mathcal U_0$ in $L^p$ (with $p,q$ as above) is compact in $L^p(\T)$, and hence weakly compact and $L^p$-equi-integrable by the Dunford-Pettis theorem, see \cite[Theorem 4.30]{Brezis}.

Hence, we just need to verify that for $c > 0$, there exists $k$ independent of $u_0 \in \mathcal U_0$ such that
\[
 \|u_0- u_0\star \Gamma(1/k)\|_{L^p(\T)} \le c.
\]
This follows again by compactness of $\mathcal U_0$ in $L^p$, which can be covered by finitely many balls $B_{c/3}(u_j)$ in $L^p(\T)$. Choosing $k$ large so that
\[
\|u_j- u_j\star \Gamma(1/k)\|_{L^p(\T)} \le c/3 \quad \text{for all $j$,}
\]
we get, for $u_0 \in B_{c/3}(u_j)$,
\begin{multline*}
\|u_0- u_0\star \Gamma(1/k)\|_{L^p(\T)} \le  \\ \|u_0- u_j\|_{L^p(\T)} +  \|u_j - u_j\star \Gamma(1/k)\|_{L^p(\T)} +  \|u_j - u_0\|_{L^p(\T)} \| \Gamma(1/k)\|_{L^1(\T)}  \le c.
\end{multline*}

\end{rem}

\begin{rem}
One can implement the same scheme to handle more general Hamilton-Jacobi equations of the form
\[
\partial_tu-\sum_{i,j}a_{ij}(x,t)\partial_{ij}u(x,t)+H(x,Du)=f(x,t)
\]
where $A\in C([0,T];W^{2,\infty}(\T))$ and $\lambda I_d\leq A\leq \Lambda I_d$ for $0<\lambda\leq \Lambda$. In particular, one has to appropriately adjust the proofs of the integral and H\"older estimates, following \cite{CG2}, and use the linear maximal regularity results in \cite{PS}.
\end{rem}

\section{Applications to Mean Field Games}\label{mfgs}
For a given couple $u_T, m_0 \in C^3(\T)$, consider the MFG system
\begin{equation*}
\begin{cases}
-\partial_t u-\Delta u+H(x,Du)=g(m(x,t))&\text{ in }Q_T\\
\partial_tm-\Delta m-\mathrm{div}(D_pH(x,Du)m)=0&\text{ in }Q_T \\
m(0) = m_0, \quad u(T)=u_T &\text{ in }\T.
\end{cases}
\end{equation*}
\subsection{The monotone (or defocusing) case}

\begin{proof}[Proof of Theorem \ref{mfg1}]
We argue that under the restrictions on $r$ it is possible to prove a priori bounds on second order derivatives of solutions to \eqref{mfgdef} (and beyond, assuming additional regularity of the data). These are typically enough to prove existence theorems. One may indeed set up a fixed-point method, or a regularization procedure, which consists in replacing $g(m)$ by $g(m\star \chi_\eps)\star\chi_\eps$ (where $\chi_\eps$ is a sequence of standard symmetric mollifiers). The existence of a solution $(m_\eps, u_\eps)$ is then standard (see e.g. \cite{Gomesbook}). Since bounds on $(m_\eps, u_\eps)$ do not depend on $\eps > 0$, it is therefore possible to pass to the limit and obtain a solution to \eqref{mfgdef}.

The key a priori bound is stated in the next Lemma \ref{estm}. Once bounds for $u$ in $W^{2,1}_q(Q_T)$, $q > (d+2)/\gamma'$, are established, one can indeed improve the estimates via a rather standard bootstrap procedure involving parabolic regularity for linear equations. Indeed, $D_p H(x,Du)$ turns out to be bounded in $L^p$ for some $p > d+2$, which is the usual Aronson-Serrin condition yielding space-time H\"older continuity of $m$ on the whole cylinder (see e.g. \cite[Theorem III.10.1]{LSU}). We can then use Theorem \ref{maxreg1} to conclude that $u$ in $W^{2,1}_q(Q_T)$ for any $q > d+2$. This immediately implies by embeddings of $W^{2,1}_q(Q_T)$ that $u$ is bounded in $C^{1+\delta,\frac{1+\delta}{2}}(Q_T)$ for any $\delta \in (0,1)$. Then, we can regard the Hamilton-Jacobi equation as a heat equation with a space-time H\"older continuous source, and by \cite[Section IV.5.1]{LSU} conclude that $u_\eps$ is bounded in $C^{2+\delta',\frac{1+\delta'}{2}}(Q_T)$ independently of $\eps$. One then goes back to the Fokker-Planck equation to deduce that $m$ also enjoys $C^{2+\delta',\frac{1+\delta'}{2}}(Q_T)$ bounds.
\end{proof}

Below we state and prove the crucial a priori estimate on solutions to \eqref{mfgdef}.

\begin{lemma}\label{estm}
Let $(u,m)$ be a classical solution to \eqref{mfgdef}. Under the assumptions of Theorem \ref{mfg1}, there exists a constant $C>0$ such that
\[
\|u\|_{W^{2,1}_q(Q_T)}+\|Du\|_{L^{\gamma q}(Q_T)} \leq C, \qquad q > \frac{d+2}{\gamma'}
\]
for some positive constant $C$ (depending only on the data).
\end{lemma}

\begin{proof} {\bf Step 1.} First order estimates. These are standard (see e.g. \cite[Proposition 6.6]{Gomesbook}), and easily obtained by testing the Hamilton-Jacobi equation with $m-m_0$ and the Fokker-Planck equation with $u-u_T$; using
the standing assumptions on $H$ and $f$ one obtains
\begin{equation}\label{foest}
\iint_{Q_T}|Du|^\gamma m \, dxdt + \iint_{Q_T} m^{r+1} \, dxdt \le C_1.
\end{equation}
\smallskip
{\bf Step 2.} Second order estimates. These are obtained by testing the Hamilton-Jacobi equation with $\Delta m$ and the Fokker-Planck equation with $\Delta u$:
\begin{multline}\label{equzz}
\iint_{Q_T}  \mathrm{Tr}(D^2_{pp}H(D^2 u)^2)m\, dxdt+\iint_{Q_T}g'(m)|Dm|^2\, dxdt\\
= \int_{\T}\Delta u_T m(T)\,dx-\int_{\T}\Delta u(0)m_0\,dx\\
-2\iint_{Q_T} \mathrm{Tr}(D^2_{px}HD^2 u)m\, dxdt-\iint_{Q_T} \Delta_xH(x,Du)m\, dxdt\ .
\end{multline}
On one hand, $\iint  \mathrm{Tr}(D^2_{pp}H(D^2 u)^2)m \ge \iint [C_H^{-1}(1+|Du|^2)^{\frac{\gamma-2}2}|D^2u|^2 -C_H]m$ by \eqref{H2}, while on the other hand by Young and Cauchy-Schwarz inequality
\begin{multline*}
-2\iint_{Q_T} \mathrm{Tr}(D^2_{px}HD^2 u)m\, dxdt-\iint_{Q_T} \Delta_xH(x,Du)m\, dxdt \le \\
\frac {C_H^{-1}} 2 \iint_{Q_T}(1+|Du|^2)^{\frac{\gamma-2}2}|D^2u|^2m +c\left( \iint_{Q_T} H(x,Du) m \,dxdt + \iint_{Q_T} m \,dxdt\right),
\end{multline*}
therefore, back to \eqref{equzz}, integrating by parts we obtain
\begin{multline*}
\frac{C_H^{-1}} 2 \iint_{Q_T}(1+|Du|^2)^{\frac{\gamma-2}2}|D^2u|^2 m \, dxdt +\iint_{Q_T}g'(m)|Dm|^2\, dxdt\\ 
= \int_{\T} \Delta u_T  m(T)\,dx-\int_{\T} u(0) \Delta m_0\,dx +c \iint_{Q_T} |Du|^{\gamma}m \,dxdt + cT.
\end{multline*}
Plugging in \eqref{foest}, using lower bounds on $g'$ and  the fact that $u \ge \min u_T + \min H(\cdot,0)$ by the comparison principle, we finally get
\begin{equation}\label{soest}
\iint_{Q_T} (1+|Du|^2)^{\frac{\gamma-2}2}|D^2u|^2 m \,dxdt + \iint_{Q_T}|D(m^{\frac{r+1}{2}})|^2 \,dxdt \le C_2.
\end{equation}
\smallskip
{\bf Step 3.} Setting $b(x,t) = -D_pH(x,Du(x,t))$, by the assumptions on $H$, \eqref{foest} and \eqref{soest} (recall also that $|D_pH(x,p)| \ge c^{-1}|p|^{\gamma-1} - c$), we have
\begin{align*}
&\iint_{Q_T}|\mathrm{div}(b)|^2m\,dxdt\leq C \quad \text{if $\gamma \le 2$} \\
&\iint_{Q_T}{|\mathrm{div}(b)|^2}{(1 + |b|)^\frac{2-\gamma}{\gamma-1}} m\,dxdt \le C \quad \text{if $\gamma > 2$},
\end{align*}
and since $m$ solves a Fokker-Planck equation with drift $b$, applying Lemma \ref{mfurtherreg} with $\mu=2$ if $\gamma \le 2$ and $\mu = \gamma$ if $\gamma > 2$ yields 
\[
\|m\|_{L^\infty(0,T;L^{\eta}(\T))}\leq C_\eta \qquad \text{ for any }1\leq \eta \le 
\begin{cases} \frac{d}{d-2} & \text{if $\gamma \le 2$} \smallskip \\
\frac{d(\gamma-1)}{d(\gamma-1)-2} & \text{if $\gamma > 2$.}
\end{cases}
\]
Therefore, using again \eqref{soest}, we have bounds on $m^{\frac{r+1}2}$ in $L^\infty(0,T;L^{\eta\frac2{r+1}}(\T)) \cap L^2(0,T;W^{1,2}(\T))$, which, by parabolic interpolation \cite[Proposition I.3.2]{DiBen} imply
\[
\|m^r\|_{L^q(\T)} \le C_2 \qquad \text{ for any } q < 
\begin{cases} 1 + \frac{d}{r(d-2)} & \text{if $\gamma \le 2$} \smallskip \\
1 + \frac{(d+2)(\gamma-1)-2}{r[d(\gamma-1)-2]} & \text{if $\gamma > 2$.}
\end{cases}
\]
Under our assumptions on $r$, the exponent $q$ can be chosen large enough to apply Theorem \ref{maxreg1}, that yields the assertion.

\end{proof}

The following lemma is needed to extract regularity information on $m$.

\begin{lemma}\label{mfurtherreg}
Let $m$ be a classical solution to 
\[
\partial_t m-\Delta m+\mathrm{div}(b(x,t)m)=0 \qquad\text{ in }Q_T\ 
\]
and assume that for $\mu \ge 2$
\[
\left(\iint_{Q_T}{|\mathrm{div}(b)|^2}{(1 + |b|)^\frac{2-\mu}{\mu-1}} m\,dxdt\right)^{\frac12}\left(\iint_{Q_T}|b|^{\frac{\mu}{\mu-1}} m\,dxdt\right)^{\frac{\mu-2}{2\mu}} \leq K\ .
\]
Then, there exists a constant $C$ depending on $K$,$\mu$ and $d$, such that\[
\|m(t)\|_{L^{p}(\T)}\leq C+\|m(0)\|_{L^{p}(\T)} \qquad \forall t \in [0,T],
\]
where $p = \frac{d(\mu-1)}{d(\mu-1)-2}$ if $d>2$, while $p\in[1,\infty)$  if $d \le 2$.
\end{lemma}

\begin{proof} The estimate can be obtained by testing the equation against $m^{p-1}$ and using parabolic interpolation (see \cite[Theorem 4.1]{Gomessub} for further details). We briefly sketch it here for completeness. Testing the equation and integrating by parts yields 
\[
\int_\T m^p(t) dx + \frac{4(p-1)}{p} \int_0^t \int_\T |D m^{\frac p2}|^2= \int_\T m^p(0) dx - (p-1)\int_0^t \int_\T\mathrm{div}(b) m^p \, dxdt.
\]
We write
\[
\int_0^t \int_\T\mathrm{div}(b) m^p \, dxdt=\int_0^t \int_\T\mathrm{div}(b)m^\frac12|b|^{\frac{2-\mu}{2(\mu-1)}}|b|^{\frac{\mu-2}{2(\mu-1)}}m^{\frac{\mu(p-1)+1}{\mu}}m^{\frac{\mu-2}{2\mu}}\,dxdt
\]
Therefore, applying generalized Holder's inequality with exponents $(2,\mu,\frac{2\mu}{\mu-2})$ we deduce
\begin{multline}\label{bo}
\int_\T m^p(t) dx + \frac{4(p-1)}{p} \int_0^t \int_\T |D m^{\frac p2}(s)|^2 \,dxds  \\ \le \int_\T m^p(0) dx + (p-1)K\left(\int_0^t \int_\T m^{\mu(p-1)+1}(s) \, dxds\right)^{\frac1\mu}.
\end{multline}
We then apply parabolic interpolation inequalities (see e.g. \cite[Proposition I.3.1]{DiBen})
\[
\|z\|_{L^{\zeta}(\T\times(0,t))}\leq C\|z\|_{L^\infty(0,t;L^\nu(\T))}^{\frac{\nu \delta}{d}}\|Dz\|_{L^{\delta}(\T\times(0,t)}\text{ with }\zeta=\delta\frac{d+\nu}{d}
\]
with $z=m^{\frac{p}{2}}$, $\delta=2$ and $\nu=\frac{d(\mu-1)}{p'}$, to deduce
\begin{multline*}
\int_0^t \int_\T m^{\mu(p-1)+1}(s) \, dxds=\int_0^t \int_\T m^{\frac{p}{2}\frac{\mu(p-1)+1}{p}2} \, dxds\\\leq C\left(\int_0^t \int_\T |Dm^{\frac{p}{2}}|^2\right)^{\frac{2}{d}}\left(\sup_{s \in [0,t]} \int_\T m^p(s) dx\right).
\end{multline*}
Then, using also Young's inequality, \eqref{bo} is less than or equal to
\begin{multline*}
 \int_\T m^p(0) dx + C\left(\int_0^t \int_\T |D m^{\frac p2}(s)|^2\,dxds \right)^{\frac1\mu}\left(\sup_{s \in [0,t]} \int_\T m^p(s) dx\right)^{\frac2{d\mu}} \\
\le \int_\T m^p(0) dx +  \frac{4(p-1)}{p} \int_0^t \int_\T |D m^{\frac p2}|^2 + C_1\left(\sup_{s \in [0,t]} \int_\T m^p(s) dx\right)^{\frac2{d(\mu-1)}}.
\end{multline*}
If $d > 2$, then $\frac2{d(\mu-1)} < 1$ and we are done. If $d \le 2$, the proof is somehow simpler, by different nature of Sobolev embeddings of $W^{1,2}(\T)$; we refer to \cite{Gomessup}.
\end{proof}

\subsection{The non-monotone (or even focusing) case}

\begin{proof}[Proof of Theorem \ref{mfg2}]
We detail only the derivation of a priori estimates for smooth solutions to \eqref{mfgdef}, i.e. the existence of a constant $C>0$ (depending only on the data) such that
\begin{equation}\label{apestu}
\|u\|_{W^{2,1}_q(Q_T)}+\|Du\|_{L^{\gamma q}(Q_T)} \leq C, \qquad q > \frac{d+2}{\gamma'}
\end{equation}
Since we fall in the maximal regularity regime for the Hamilton-Jacobi equation, the existence of a solution to the MFG system can then be derived as in Theorem \ref{mfg1}.

We start from first order estimates as in Lemma \ref{estm}, that now yield, using the assumptions on $g$,
\[
\iint_{Q_T}L\big(x,D_pH(x,Du)\big)m \, dxdt - \iint_{Q_T} m^{r+1} \, dxdt \le C_1.
\]
To bound the two terms in the left hand side separately, a further step is needed. Recall a crucial Gagliardo-Nirenberg type inequality proven in \cite[Proposition 2.5]{CT}, that reads
\begin{equation}\label{gnct}
\left(\int_{Q_T} m^{r+1}(x,t) \, dxdt \right)^\delta \le C\left(\int_{Q_T} |D_p H(x, Du(x,t))|^{\gamma'} \, dxdt + 1\right).
\end{equation}
for some $C > 0$ and $\delta < 1$, provided that $r < \gamma' / d$. Then, using this inequality, the assumptions on $L$ and the fact that $\int_\T m(t) = 1$ for all $t$,
\begin{multline*}
C_L^{-1}\int_{Q_T} |D_p H(x, Du(x,t))|^{\gamma'} - C_L T\ \\ \le \iint_{Q_T}L\big(x,D_pH(x,Du)\big)m \, dxdt \le C\left(\int_{Q_T} |D_p H(x, Du(x,t))|^{\gamma'} \, dxdt + 1\right)^{1/\delta} +  C_1.
\end{multline*}
Hence, back to \eqref{gnct}, we get
\begin{equation}\label{foest3}
\|m\|_{L^{r+1}(Q_T)}  \le C_2.
\end{equation}
Therefore, $\|g(m)\|_{L^{\frac{r+1}{r}}(Q_T)}  \le C_2$. Note that by the assumptions on $r$, $q = \frac{r+1}r$ is large enough  to apply Theorem \ref{maxreg1}, and \eqref{apestu} follows.

\end{proof}

\appendix
\section{Some embedding theorems}

\begin{lemma}\label{embprel}
For $p > 1$, the space $\H_p^1(Q_T)$ is continuously embedded into $C([0,T];W^{1-2/p,p}(\T))$, and $W^{2,1}_p(Q_T)$ is continuously embedded into $C([0,T];W^{2-2/p,p}(\T))$. 
\end{lemma}
\begin{proof}
We consider the embedding of $\H_p^1$ only (the other can be obtained similarly). Recall that the case $p=2$ is classical (see e.g. \cite[Theorem XVIII.2.1]{DL}). The general statement can be proven via abstract methods for evolution problems, see \cite{Amann,PS}. We provide a short proof for reader's convenience.
First, $u\in\H_p^1(Q_T)$ can be extended to a $v \in \H_p^1(\T \times (0,\infty))$ in the usual way: let $u(t) = u(T)$ for all $t \ge T$, and set $v(t)=\zeta(t)u(t)$, where $\zeta$ is  a smooth function in $(0,\infty)$ which vanishes for $t\geq T+1$ and is identically one for $t\in[0,T]$. Then, since $\H_p^1(\T\times(0,+\infty)) \simeq W^{1,p}(0,+\infty; (W^{1,p'}(\T))^{'})\cap L^p(0,+\infty;W^{1,p}(\T))$, apply \cite[Corollary 1.14]{Lunardi} to obtain that
\[
\H_p^1(\T\times(0,+\infty))\hookrightarrow C_b([0,+\infty);(W^{-1,p}(\T),W^{1,p}(\T))_{1-1/p,p}).
\]
One then concludes by means of the Reiteration Theorem \cite[Theorem 1.23]{Lunardi} (see also \cite{Leoni}) that \[(W^{-1,p}(\T),W^{1,p}(\T))_{1-1/p,p}\simeq W^{1-2/p,p}(\T),\] which gives the statement.\\
\end{proof}

\begin{prop}\label{emb1}
For $p > 1$, the parabolic space $\H_p^{1}(Q_T)$ into $L^{\delta}(0,T;W^{\alpha,\delta}(\T))$, where $\delta > p$ and
\[
\alpha = 1+\frac{d+2}{\delta}-\frac{d+2}{p}.
\]
\end{prop}

\begin{proof}
We adapt a strategy presented in \cite{k1,k2} (see also \cite{CG1} for the Bessel potential spaces setting). Let $\theta=p/\delta \in (0,1)$ and $\nu =(1-2/p)(1-\theta)+\theta$. We now use the (real) interpolation in the Sobolev-Slobodeckij scale to observe that $W^{\nu,p}(\T)$ can be obtained by interpolation between $W^{1,p}(\T)$ and $W^{1- 2/p,p}(\T)$ (see \cite[Theorem 2.4.2 p.186 and eq. (16)]{trbookinterpolation}). Moreover, $W^{\nu,p}(\T)$ is continuously embedded into $W^{\nu+d/\delta-d/p,\delta}(\T)$, see \cite{ST}. Hence, for a.e. $t$,
\begin{equation*}
c(d, p, \delta)\norm{u(t)}_{W^{\nu-\frac{d}{p}+\frac{d}{\delta},\delta}(\T)}\leq\norm{u(t)}_{W^{\nu,p}(\T)}\leq \norm{u(t)}_{W^{1- 2/p,p}(\T)}^{1-\theta}\norm{u(t)}_{W^{1,p}(\T)}^{\theta}.
\end{equation*}
Then, $\alpha = \nu-\frac{d}{p}+\frac{d}{\delta} = 1+\frac{d}{\delta}-\frac{d+2(1-\theta)}{p}$ and 
\begin{multline*}
\left(\int_0^T\norm{u(t)}_{W^{\alpha,\delta}(\T)}^{\frac{p}{\theta}}dt\right)^{\theta}\leq
C_1\left(\int_0^T\norm{u(t)}_{W^{1- 2/p,p}(\T)}^{(1-\theta)\frac{p}{\theta}}\norm{u(t)}_{W^{1,p}(\T)}^pdt\right)^{\theta} \\
 \leq C_2\sup_{t \in [0, T]}\norm{u(t)}_{W^{1- 2/p,p}(\T)}^{(1-\theta)p}\left(\int_0^T\norm{u(t)}_{W^{1,p}(\T)}^pdt\right)^{\theta},
\end{multline*}
Recalling now that $\theta=p/\delta \in (0,1)$ we obtain
\[
\left(\int_0^T\norm{u(t)}_{W^{\alpha,\delta}(\T)}^{\delta}dt\right)^{\frac1\delta}\leq C_3\|u\|_{\H_p^1(Q_T)}.
\]
\end{proof}
%\begin{rem}
%The result in Proposition \ref{emb1} can be generalized to mixed classes $L^{q_2}(0,T;W^{\alpha,q_1}(\T))$ when 
%\[
%\alpha=1-\frac{d+2}{p}+\frac{d}{q_1}+\frac{2}{q_2}\ .
%\]
%This can be achieved by using the embeddings in \cite[Section 5.5]{PrussSimonett}.
%\end{rem}
\begin{lemma}\label{embnik}
For $0 < \alpha < 1$ and $1\leq p<\infty$, $W^{\alpha,p}(\T)$ is continuously embedded into $N^{\alpha,p}(\T)$.
\end{lemma}

\begin{proof} We just need to localize analogous results on $\R^d$, which go back to \cite[Lemma 9 p. 441]{T}, see also \cite[Proposition 10-(b) Section 5.2]{Stein} and \cite[Theorem 17.38]{Leoni} for a proof via real interpolation methods. Let $\chi$ be a compactly supported cut-off function such that $\chi\equiv 1$ on the unit cube $[-2,2]^d$.  It is immediate to see that the extension operator is bounded
\[
W^{k,p}(\T) \ni u\longmapsto \tilde u=\chi u \in W^{k,p}(\R^d)
\]
for all nonnegative integers $k\geq0$ and $p\geq1$, so it is bounded from $W^{\alpha,p}(\T)$ to $W^{\alpha,p}(\R^d)$ by interpolation. 
Then,
\[
\|u\|_{N^{\alpha,p}(\T)}\leq C_1\|\tilde u\|_{N^{\alpha,p}(\R^d)}\leq C_2\|\tilde u\|_{W^{\alpha,p}(\R^d)}\leq C_3\|u\|_{W^{\alpha,p}(\T)}\ ,
\]
where the second inequality relies on the aforementioned embedding $W^{\alpha,p}(\R^d) \hookrightarrow N^{\alpha,p}(\R^d)\simeq B_{p\infty}^\alpha(\R^d)$.
\end{proof}

%\bibliography{par_max}
%\bibliographystyle{abbrv}

\medskip

\begin{flushright}
\noindent \verb"cirant@math.unipd.it"\\
%Dipartimento di Matematica ``Tullio Levi-Civita'' \\
%Universit\`a di Padova\\
%via Trieste 63, 35121 Padova (Italy)\\
%\smallskip
\noindent \verb"alessandro.goffi@math.unipd.it"\\
Dipartimento di Matematica ``Tullio Levi-Civita'' \\
Universit\`a di Padova\\
via Trieste 63, 35121 Padova (Italy)
\end{flushright}

\end{document}